\documentclass[11pt]{amsart}
\usepackage{a4wide}

\usepackage{textcmds} 
\usepackage{amsmath, amssymb, amsfonts, amstext, amsthm, amscd, mathrsfs}
\usepackage{mathtools}
\usepackage{verbatim}
\usepackage{graphicx}
\usepackage{color}
\usepackage{pinlabel}
\usepackage{mathrsfs} 
\usepackage{xypic}
\usepackage{enumitem}
\usepackage{tikz}
\usetikzlibrary{matrix}

\usepackage{etoolbox}
\makeatletter
\let\ams@starttoc\@starttoc
\makeatother
\usepackage[parfill]{parskip}
\makeatletter
\let\@starttoc\ams@starttoc
\patchcmd{\@starttoc}{\makeatletter}{\makeatletter\parskip\z@}{}{}
\makeatother

\usepackage{hyperref}

\newcommand{\kk}{\mathbf{k}}

\newcommand{\R}{\mathbb{R}}
\newcommand{\Q}{\mathbb{Q}}
\newcommand{\Z}{\mathbb{Z}}

\newcommand{\id}{\mathrm{Id}}
\newcommand{\im}{\mathrm{im}}

\newcommand{\OP}{\operatorname}

\newcommand{\pt}{\operatorname{pt}}

\newsavebox{\textvisiblespacebox}
\savebox{\textvisiblespacebox}{\texttt{aa}}
\newcommand\vartextvisiblespace[1][\wd\textvisiblespacebox]{%
  \makebox[#1]{\kern.1em\rule{.4pt}{.3ex}%
  \hrulefill%
  \rule{.4pt}{.3ex}\kern.1em}%
}

\numberwithin{equation}{section}

\newtheorem{thm}{Theorem}[section]
\newtheorem{lma}[thm]{Lemma}
\newtheorem{prp}[thm]{Proposition}

\newtheoremstyle{TheoremNum}
    {\topsep}{\topsep}              
    {\itshape}                      
    {}                              
    {\bfseries}                     
    {.}                             
    { }                             
    {\thmname{#1}\thmnote{ \bfseries #3}}
\theoremstyle{TheoremNum}

\theoremstyle{definition}
\newtheorem{dfn}[thm]{Definition}

\newtheorem{ex}[thm]{Example}

\theoremstyle{remark}
\newtheorem{rmk}[thm]{Remark}

\begingroup 
\makeatletter 
\@for\theoremstyle:=definition,remark,plain,TheoremNum\do{%
\expandafter\g@addto@macro\csname th@\theoremstyle\endcsname{%
\addtolength\thm@preskip\parskip 
}%
} 
\endgroup 

\title[The persistence of the Chekanov--Eliashberg algebra]{The persistence of the Chekanov--Eliashberg algebra}

\author{Georgios Dimitroglou Rizell}
\address{Department of Mathematics\\
Uppsala University\\
Box 480\\
SE-751 06 Uppsala\\
Sweden}
\email{georgios.dimitroglou@math.uu.se}

\author{Michael G. Sullivan}
\address{Department of Mathematics and Statistics\\
University of Massachusetts\\
Amherst\\
MA 01003\\
USA}
\email{sullivan@math.umass.edu}

\begin{document}

\begin{abstract}

We apply the barcodes of persistent homology theory to the Chekanov-Eliashberg algebra of a Legendrian submanifold to deduce displacement energy bounds for arbitrary Legendrians. We do not require the full Chekanov-Eliashberg algebra to admit an augmentation as we linearize the algebra only below a certain action level. As an application we show that it is not possible to $C^0$-approximate a stabilized Legendrian by a Legendrian that admits an augmentation.
\end{abstract}

\maketitle

\section{Introduction}
\label{sec:Introduction}

The (Lagrangian) Arnold conjecture states that the number of intersection points of a Lagrangian submanifold with its Hamiltonian image is bounded below by the sum of the Lagrangian's Betti numbers.
Floer developed Lagrangian Floer theory to prove this bound in certain cases \cite{Floer:MorseTheoryLagrangian}, but the bound does not always hold. Chekanov, using Hofer's norm \cite{HoferNorm} for the Hamiltonian isotopy, measured how ``long" this bound persists by measuring how long Floer theory remains valid: how long $d^2=0$ holds, and how long the Floer theory remains invariant \cite{Chekanov}. Lagrangian Floer theory also provides a similar temporary lower bound on the number of Reeb chords between a Legendrian and its contact Hamiltonian image \cite{DimitroglouRizellSullivan}.

In this article, we study the persistence of Reeb chords between a Legendrian submanifold $\Lambda \subset (Y,\ker \alpha)$ of a contact manifold and its image under a contact isotopy. We replace Floer theory with the linearized Chekanov--Eliashberg algebra $\mathcal{A}(\Lambda)$ (also called the Legendrian contact DGA) induced by a choice of augmentation. We again measure things in terms of the Hofer norm of the contact Hamiltonian isotopy.
In addition to measuring how long a certain linearized Chekanov--Eliashberg homology theory remains well-defined, we also measure {\em{how much}} of it persists.

We require that $\mathcal{A}(\Lambda)$ be rigorously well-defined, that the ``handle-slide and birth/death bifurcation-analysis" (Section \ref{sec:HSBD}) proof of invariance holds, and (for Theorem \ref{thm:main}) that there is a certain correspondence of $J$-holomorphic disks (Proposition \ref{prp:ModuliSpaceLift}). 
As of this writing, these requirements restrict our ambient contact manifold to be $(Y, \xi) = (P \times \R_z, \ker{\alpha}).$ Here $(P, d\lambda)$ is an exact symplectic manifold tame at infinity (Gromov compactness holds), and the contact 1-form, which determines the Reeb flow, must be of the type $\alpha_{\OP{std}} \coloneqq dz + \lambda$ \cite{EES07}. Note that this includes one-jet spaces endowed with the canonical contact form.

\subsection{Background}

A {\bf Reeb chord} on a Legendrian submanifold is a non-trivial integral curve of the Reeb vector field $R_\alpha \in \Gamma(TY)$ defined by $\iota_{R_\alpha}\alpha=1$ and $\iota_{R_\alpha}d\alpha=0.$ We are interested in estimating the number of Reeb chords from a given Legendrian $\Lambda$ to its image under a contact isotopy. If there are no such Reeb chords, we say that the contact isotopy {\bf displaces} $\Lambda.$ Of course this notion depends on the choice of contact form.

The set of $\alpha$-Reeb chords of $\Lambda$ 
 canonically generates 
$\mathcal{A}(\Lambda)$ as a free noncommutative algebra. The grading is a certain Maslov index. 
The differential $\partial$ counts $J$-holomorphic disks in the symplectization $(\R_t \times Y, d(e^t \alpha))$ with Lagrangian boundary condition $\R_t \times \Lambda.$ 
The homotopy-type of the DGA is invariant under {\bf{Legendrian isotopy}}, which is a smooth isotopy of Legendrian submanifolds. We often notationally suppress the grading, the differential $\partial,$ the dependence on $J$ and $\alpha.$ See \cite{EES07} and references therein for definitions.

Each Reeb chord $c$ has a {\bf{length}} (also called {\bf{action}}) $ \ell(c) := \int_c \alpha.$
For $0 < l \le \infty,$
let $\mathcal{A}^l(\Lambda)$ be the unital sub-algebra generated by those generators with length bounded from above by $l.$
The action preserving properties of the differential of the Chekanov--Eliashberg algebra implies that $\mathcal{A}^l(\Lambda) \subset \mathcal{A}(\Lambda)$ is a unital sub-DGA. To that end, recall that the differential applied to a generator $c$ consists of a sum of words of generators whose lengths all are strictly less than the length of $c,$ as follows from an elementary application of Stokes' theorem.

An {\bf{augmentation}} for the DGA $\mathcal{A},$ $\varepsilon: (\mathcal{A}, \partial) \rightarrow (\kk,\partial_\kk := 0),$ is a (graded) DGA-morphism
to the ground field $\kk$ viewed as a DGA. 
We will want to choose $l$ such that $\mathcal{A}^l(\Lambda)$ has an augmentation; since $\mathcal{A}^l=\kk$ for $l>0$ sufficiently small, this is always possible.
If $\Lambda$ is loose in the sense of Murphy \cite{Murphy:Loose} and $c$ is the Reeb chord in a standard representative of the loose chart then there are a number of standard Legendrian contact homology arguments which show, up to unit, $\partial(c)=1.$ The contradiction
$$ \varepsilon \circ \partial (c) = \varepsilon (1)=1 \ne 0 = \partial_\kk \circ \varepsilon(c),$$
 means that we cannot choose $l\ge\ell(c).$

 Henceforth assume all Legendrians are compact and all isotopies are compactly supported. 
The {\bf{oscillation}} of a contact Hamiltonian $H_t:Y \times \R_t \rightarrow \R$ 
\[
\|H_t\|_{\OP{osc}}^Y 
\coloneqq \int_0^1 \left(\max_{y \in Y} H_t - \min_{ y \in Y} H_t\right)dt
\]
is the key ingredient in the Hofer norm of the corresponding contact Hamiltonian isotopy $\phi^t_{\alpha,H_t}$ (which is defined by
$ H_t(\phi^t_{\alpha,H_t}(x)) = \alpha\left(\frac{d}{dt}\phi^t_{\alpha,H_t}(x)\right)$). Contact isotopies are generated by uniquely-defined contact Hamiltonians, that depend only on the choice of contact form; we thus sometimes say the oscillation of a contact isotopy when we mean the oscillation of the corresponding generating contact Hamiltonian.
This article focuses on contact isotopies acting on Legendrian submanifolds. These are Legendrian isotopies.
For a Legendrian isotopy $\varphi^t \colon \Lambda \hookrightarrow Y$ we can consider the induced family of smooth functions
$$g_t \coloneqq \alpha\left(\frac{d\varphi^t}{dt}\right)\colon \Lambda \to \R,$$
that we identify with a family of functions $h_t \colon \varphi^t(\Lambda) \to \R$ defined on the subsets $\varphi^t(\Lambda) \subset Y.$ 
Define the {\bf{oscillation}} of a Legendrian isotopy to be
\[
\|h_t\|_{\OP{osc}}^\Lambda
\coloneqq \int_0^1 \left(\max_{y \in \varphi^t(\Lambda)} h_t - \min_{ y \in \varphi^t(\Lambda)} h_t\right)dt
\]

The perspectives between Legendrian isotopies $\varphi^t$ and ambient contact isotopies $\phi^t_{\alpha,H_t}$ can be switched. By a smooth restriction of $H_t$ to $\Lambda$ and smooth extension of $h_t$ to $P \times \R$ we obtain a bijective correspondence 
$$\phi^t_{\alpha,H_t}(\varphi^0(\Lambda))=\varphi^t(\Lambda)
\longleftrightarrow H_t|_{\varphi^t(\Lambda)}=h_t.
$$
The smooth extension can be chosen arbitrarily, and two different such extensions only results in different \emph{parametrizations} of the Legendrian isotopy. Conversely, changing the Legendrian isotopy $\varphi^t$ by the precomposition $\varphi^t \circ \psi^t$ of a smooth isotopy $\psi^t \colon \Lambda \to \Lambda$ leaves $h_t$ unchanged. See e.g.~the proof of \cite[Theorem 2.6.2]{Geiges}. 
Furthermore, in the case when $h_t$ is {\bf{indefinite}}, i.e. $h_t^{-1}(0) \neq \emptyset,$ there exists a smooth extension to a globally defined compactly supported contact Hamiltonian $H_t$ such that
$\|H_t\|^Y_{\OP{osc}} = \|h_t\|^\Lambda_{\OP{osc}}.$ Indeed, we can extend $h_t$ to $H_t$ by first making it constant in each fiber of the normal bundle of the Legendrian, and then cutting off this local extension by the multiplication with a suitable smooth cut-off function. 

By using \cite[Lemma 2.3]{DimitroglouRizellSullivan} we may replace $h_t$ by a contact Hamiltonian that is indefinite without changing its oscillatory norm. This changes the isotopy by a translation of the $z$-coordinate, which is irrelevant for the displaceability questions that we study here.

Henceforth, we abandon these distinctions and refer to {\em{the}} oscillation, using the compromise notation of $\|H_t\|_{\OP{osc}}.$

\subsection{Results}

Fix $1 \le k \le n = \dim\Lambda.$ 
Given a Reeb chord $c,$ let $\mathcal{M}^{k}(c)$ denote the formal-dimension $k$ moduli spaces of $J$-holomorphic disks $(D, \partial D) \rightarrow (Y \times \R, \Lambda \times \R),$ with one positive boundary puncture at $c$
and $m \ge 0$ negative boundary punctures for a cylindrical almost complex structure $J$. (For $k=1$ these are exactly the moduli spaces used to define $\partial(c)$ in $\mathcal{A}(\Lambda)$ \cite{EES07}.) 
Let
$$
\tilde{\sigma}_k \coloneqq \min_{c|\mathcal{M}^{k}(c) \ne \emptyset} \ell(c) \le \infty.
$$ 
Let ${\sigma}_k \coloneqq \min\{\tilde{\sigma}_k, \tilde{\sigma}_{n-k}\}$ for $1 \le k \le n-1$ and let 
${\sigma}_k \coloneqq \tilde{\sigma}_n$ 
for $k=0,n.$
 Note that $\sigma_k=\sigma_{n-k}.$
Let $\{ \iota_0, \iota_1, \ldots, \iota_n\}$ denote an ordering such that $\sigma_{\iota_i} \ge \sigma_{\iota_{i+1}}.$ 
\begin{thm}
\label{thm:main}
 Suppose that $\mathcal{A}^l(\Lambda)$ with $0 < l \le \infty$ admits an augmentation into the field $\kk.$ 

For any compactly supported contact Hamiltonian $H_t \colon P \times \R \to \R,$ suppose that the inequality
$$\|H_t\|_{\OP{osc}} < \min\left\{l,\sigma_{\iota_i}\right\}$$ 
is satisfied for some $i \in \{0,\ldots,n\},$ and that $\phi^1_{{\alpha_{\OP{std}}},H_t}(\Lambda)$ is transverse to the Reeb flow applied to $\Lambda.$ Then there exists at least a number
$$\sum_{j=0}^{i} \dim\left(H_{\iota_j}(\Lambda; \kk)\right)$$
of Reeb chords with one endpoint on $\Lambda,$ one endpoint on $\phi^1_{\alpha,H_t}(\Lambda).$
Furthermore, these Reeb chords are all of length less than $\|H_t\|_{\OP{osc}}.$
\end{thm}


\begin{rmk}
If $\Lambda$ is spin and orientable, we can define the Chekanov-Eliashberg algebra with $\Q$ or $\Z_p$-coefficients for $p$ prime \cite{EES05c}.
Our arguments work for these fields, 
so we can set $\kk$ to equal $\Q$ or $\Z_p$ in the above bound. 
Without this assumption on $\Lambda,$ we set $\kk$ equal to $\Z_2.$
\end{rmk}

\begin{rmk}
\label{rmk:pointconstraints}
Fix a generic choice of points $\pt_i \in \Lambda$ disjoint from the Reeb chords, one for each connected component of $\Lambda.$
 If we wish to, for the definition of $\tilde{\sigma}_n,$
 we can impose the additional requirement that 
 some disk in $\mathcal{M}^{n}(c)$ must have its boundary pass
 through the union of rays $\R \times \{\pt_1,\ldots,\pt_{\pi_0(\Lambda)}\}.$
\end{rmk}

Recall the {\bf{standard Legendrian sphere}} in $\R^{2n+1}$ whose front projection in $\R^{n+1}_{x_1, \ldots, x_n,z}$ has an $O(n)$-symmetry in the $x_i$-directions. It has one Reeb chord $c$ along the $z$-axis of Legendrian Contact Homology index $n$ and its Thurston-Bennequin invariant is $\tt{tb}=(-1)^{(n^2+n+1)/2}$. The front when $n=2$ is depicted in Figure \ref{fig:standardsphere}. See \cite[ Section 4.3]{EES05b} for a review of these concepts.
When $\ell(c) = a$ we denote this sphere by $\Lambda_{\OP{St}}(a).$

\begin{figure}[htp]
\vspace{5mm}
\centering
\labellist
\pinlabel $z$ at 44 63
\pinlabel $a/2$ at 53 50
\pinlabel $-a/2$ at 55 7
\pinlabel $x_1$ at 20 2
\pinlabel $x_2$ at 100 27
\pinlabel $\Lambda_{\OP{St}}(a)$ at 73 40
\pinlabel $\color{red}c$ at 40 38
\endlabellist
\includegraphics[scale=1.5]{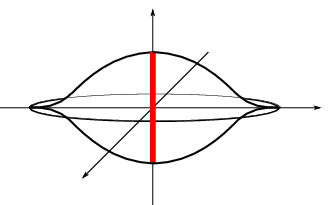}
\vspace{3mm}
\caption{The front of the standard Legendrian two-sphere.}
\label{fig:standardsphere}
\end{figure}

\begin{ex}
\label{ex:standard}
For the standard Legendrian $n$-dimensional sphere $\Lambda_{\OP{St}}(a) \subset \R^{2n+1}$ we have $l=+\infty$ and
$$\sigma_k=\begin{cases} +\infty, & 0< k <n,\\ a, & k=0,n. \end{cases}$$
 (See e.g.~\cite{NonIsoLeg}.) Theorem \ref{thm:main} shows that a contact Hamiltonian that displaces $\Lambda_{\OP{St}}(a)$ must be of oscillation at least $a.$ 
\end{ex}

\begin{ex}
\label{ex:stabilized}
For the stabilized unknot $\Lambda_{\OP{Stab}}$
shown in Figures \ref{fig:stabilized} and \ref{fig:stabilizedproj}, $l=\min\{\ell(c_1),\ell(c_2)\}.$ If we take the point constraint $\pt \in \Lambda_{\OP{Stab}}$ for the definition of $\sigma_0=\sigma_1$ as in Remark \ref{rmk:pointconstraints} to live near the crossing of the front projection of $\Lambda_{\OP{Stab}},$ we further get $\sigma_0 = \sigma_1 = +\infty$; there are no pseudoholomorphic discs with one positive punctures passing through that region.
\end{ex}

\begin{figure}[htp]
	\vspace{3mm}
	\labellist
	\pinlabel $z$ at 8 100
	\pinlabel $\color{red}c_1$ at 68 46
	\pinlabel $\color{red}c_2$ at 159 46
	\pinlabel $x$ at 233 7
	\endlabellist
	\includegraphics{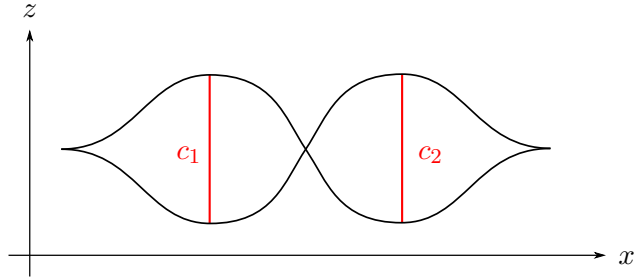}
	\caption{The front projection of the stabilized unknot $\Lambda_{\OP{Stab}}$.}
	\label{fig:stabilized}
\end{figure}

\begin{figure}[htp]
	\labellist
	\pinlabel $c_1$ at 59 42
	\pinlabel $c_2$ at 167 42
	\pinlabel $x$ at 233 0
	\endlabellist
	\includegraphics{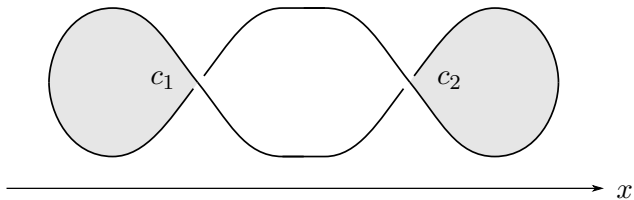}
	\caption{
	The Lagrangian projection of the stabilized unknot $\Lambda_{\OP{Stab}}.$ The pseudoholomorphic disks with one positive puncture are shaded.}
	\label{fig:stabilizedproj}
\end{figure}

\begin{lma}
\label{lma:sharp}
 For any closed Legendrian $\Lambda \subset \{|z| \le a/2\} \subset (\R^{2n+1},\alpha_{\OP{std}})$ and $\epsilon>0,$ there exists a contact Hamiltonian of oscillatory norm $a+\epsilon$ that displaces $\Lambda.$ In particular, Theorem \ref{thm:main} is sharp for Example \ref{ex:standard}
\end{lma}

Theorem \ref{thm:main} result follows from a more technical Barcode Proposition \ref{prp:Barcode}, adapted from persistence homology literature.
The following non-squeezing result also follows from the Barcode Proposition.

Consider an open subset $U \subset Y^{2n+1}$ having the homotopy type of a closed $n$-manifold. We say that a Legendrian $\Lambda \subset Y$ can be {\bf squeezed} into $U$ if there exists a contact isotopy $\phi^t$ of the ambient space such that
\begin{itemize}
\item $\phi^1(\Lambda) \subset U;$ and
\item $[\Lambda] \in H_n(U;\Z_2)$ is nonzero.
\end{itemize}

Stabilized Legendrians \cite[Section 3]{EES05b} of dimension at least two are loose
and satisfy an $h$-principle due to Murphy \cite{Murphy:Loose}; in particular, they are $C^0$-dense in the space of Legendrian embeddings. By this $h$-principle, a loose Legendrian 
$\Lambda'$ 
can be squeezed into $U$ whenever this subset contains a formal Legendrian embedding in the same formal Legendrian isotopy class as $ \Lambda'.$

\begin{thm}
\label{thm:NonSqueeze}
Let $\Lambda_{\OP{stab}} \subset \R^{2n+1}$ be a stabilized Legendrian, and suppose for $\Lambda' \subset \R^{2n+1}$ that $\mathcal{A}(\Lambda')$ admits an augmentation. Then $\Lambda'$ cannot be squeezed into a standard contact neighborhood of $\Lambda_{\OP{stab}}.$
\end{thm}

There has been recent interest in $C^0$-limits of Hamiltonian diffeomorphisms (contact or symplectic) \cite{BuhovskyHumiliereSeyfaddini, RosenZhang}.
In particular, Rosen and Zhang prove that for the $C^0$-limit of the images of a Legendrian submanifold under a sequence of contactomorphisms, if the limit is smooth and the sequence satisfies \cite[Definition 4.1]{RosenZhang}, then the limit is also Legendrian. The definition in particular requires the sequence of contactomorphisms to have uniformly converging conformal factors. In \cite[Remark 1.5]{RosenZhang} the authors suggest that the latter mechanism prevents such a sequence from being an ``approximation by zig-zags'' (recall that e.g.~any, not even necessarily Legendrian, knot can be approximated by a Legendrian knot by adding more and more zig-zags). Theorem \ref{thm:NonSqueeze} could be interpreted as evidence that something stronger actually prevents this, as it shows that a stabilized Legendrian cannot be approximated by any Legendrian that admits an augmentation.

The homology condition in the definition is crucial. In $\R^{2n+1}$ all Legendrians can be put in a neighborhood of any other Legendrian by first rescaling it to make it sufficiently small. Also, the 2-copy of $\Lambda_{\OP{stab}}$ sitting in an arbitrary small neighborhood of $\Lambda_{\OP{stab}},$ \ has non-zero $\Z$-homology but vanishing $\Z_2$-homology. The 2-copy has an augmentation which one can see either by an explicit construction of an exact Lagrangian filling by a cylinder $I \times \Lambda_{\OP{stab}}$ or, alternatively, by an explicit calculation of the Chekanov--Eliashberg algebra of the two-copy using the theory from \cite{Ekholm:Duality}.

Using this $h$-principle to approximate a non-Legendrian deformation of the initial Legendrian by the stabilized version we prove Theorem \ref{thm:DisplaceLoose} below, 
in contrast to Theorem \ref{thm:main} above.

\begin{thm}
\label{thm:DisplaceLoose}
Consider the 2-sphere $\Lambda_{\OP{St}}(1) \subset \R^5.$
Fix any $1 \gg \delta >0.$
After making the sphere loose by adding a stabilization contained inside a sufficiently small neighborhood of an arbitrary point in $\Lambda_{\OP{St}}(1),$ it is displaceable by a contact Hamiltonian of oscillation less than $\delta.$
\end{thm}
Our proof is specific to the this Legendrian $\Lambda_{\OP{St}}(1).$ However, we expect that the techniques can be adapted to show the statement also for arbitrary Legendrians.

\subsection{Related results}
Persistence homology has made some recent inroads in symplectic geometry, some of which we list here. 
 Polterovich and Shelukhin introduced the technique of persistence homology in Floer homology in their work \cite{Polterovich:Persistence}, which included measuring the bottle-neck distance between Hamiltonian Floer complexes. This was later generalized to the Floer-Novikov complex which by Usher and Zhang \cite{UsherZhang}.
Polterovich, Shelukhin and Stojisavljevic introduce a new structure to persistence homology for Floer theory, coming from the quantum cup product \cite{PolterovichShelukhinStojisavljevic}.
Buhovsky, Humiliere and Seyfaddini \cite{BuhovskyHumiliereSeyfaddini} 
apply the barcodes technology of Kislev and Shelukhin
\cite{KislevShelukhin} to study Hamiltonian (non-smooth) homeomorphisms.

All of the above are in the symplectic setting, focusing on quantitative results for Hamiltonians.
There are fewer quantitative results for Reeb dynamics.

Entov and Polterovich prove that there is a Reeb chord between a Legendrian and its image under the Reeb flow of a different contact form that defines the same co-orientated contact structure, assuming that the Hamiltonian restricted to the Legendrian has a small oscillation \cite[Remark 1.14]{EntovPolterovich}.
They are currently developing a barcode approach to this result. 
Their chord can be viewed as a flow from the Legendrian to itself gotten by taking the Reeb flow first followed with the contact flow, whereas our chords in Theorem \ref{thm:main} are gotten by taking the contact flow first followed with the Reeb flow.
We prove a lower bound on the number of Reeb chords similar to that of Theorem \ref{thm:main} \cite{DimitroglouRizellSullivan}. In this earlier article, 
our bound uses the conformal factor and is like Chekanov's bound \cite{Chekanov}: the lower bound is either the full sum of Betti numbers, or 0 (we do not measure its persistence). In the symplectic setting a similar refined count of intersection points was obtained in \cite{FukayaDisplacement} by Fukaya, Oh, Ohta and Ono.


\subsection{Acknowledgments}
We would like to thank Michael Entov for the above comparison of results and Lenny Ng for Remark \ref{rmk:Lenny}. We are also in debt to the referee for doing a very careful reading, finding mistakes, and giving valuable suggestions. The first author is supported by grant KAW 2016.0198 from the Knut and Alice Wallenberg Foundation.
The second author is supported by grant 317469 from the Simons Foundation and thanks the Centre de Recherches Mathematiques for hosting him while some of this work was done.

\section{Barcodes of filtered complexes with action windows}
\label{sec:Barcodes}
We consider an abstract setup of filtered complexes that we later will apply to the setting of Legendrian contact homology. First we define a notion of a piecewise continuous (PWC) family of filtered complexes with action window parametrized by time $t \in \R$. Roughly speaking, the complexes have a continuously varying finite action window and the generators of the complex can leave and enter this window. Observe that these bifurcations need not preserve the homology of the complex. We adapt the theory of barcodes from persistent homology theory (see e.g., \cite[Section 1]{UsherZhang} and its references) to this setting. 

Throughout this section we will always assume our complexes to be finite-dimensional $\kk$-vector spaces.

\subsection{PWC families of complexes with action window}
\label{sec:ContFamily}


Fix $t \in \R. $ A {\bf filtered complex} $C(t)^{b_t}_{a_t}$ {\bf with action window $[a_t,b_t)\subset \R$} is a finite-dimensional $\kk$-complex with action function
$$\ell_t \colon C(t)^{b_t}_{a_t} \to [a_t,b_t) \cup \{-\infty\}$$
that induces the filtration by subcomplexes of the form $\{x; \:\: \ell_t(x) < c\} \subset C(t)^{b_t}_{a_t}$ for any number $c \in \R$.

Given any sub-interval $[a'_t,b'_t) \subset [a_t,b_t)$ there is an induced filtered complex 
$$C(t)^{b'_t}_{a'_t} \coloneqq \{x \in C(t)^{b_t}_{a_t};\:\: \ell(x)<b'_t\}/\{x \in C(t)^{b_t}_{a_t};\:\: \ell(x)<a'_t\}$$
with action window $[b'_t,a'_t)$ obtained as a suitable quotient complex.

A basis $e_1,\ldots,e_m$ is {\bf compatible with the filtration} if the action filtration on any arbitrary element is given by the formula
$$\ell_t(r_1e_1+\ldots+r_me_m)=\max_{r_i \neq 0}\ell_t(e_i)$$
in terms of the actions of the basis elements.

\begin{rmk}In our application, the complexes are endowed with a geometrically induced compatible basis which consists of Reeb chords, whose elements have action equal to their Reeb chord lengths.
\end{rmk}

\begin{lma}[Barannikov decomposition]
\label{lma:compatiblebasis}
There always exists a (non-canonical) choice of basis which is compatible with the filtration. Furthermore, given any compatible basis, there exists an action-preserving change of basis (i.e.~which is upper-triangular when the basis elements are ordered by non-decreasing action) such that the following holds.
\begin{itemize}
\item A subset of the basis yields a basis for the cycles and, thus, the differential applied to its complementary subset yields a basis for the boundaries.
\item The differential of a basis element is either zero or itself a basis element.
\end{itemize}
\end{lma}
\begin{proof}

Consider the filtration $C^c = \{ x; \:\: \ell_t(x) < c\}$ by subcomplexes and the inclusion maps

$$\iota_{c,c+\epsilon} \colon C^{c} \to C^{c+\epsilon}.$$
For all but finitely many action levels $c \in \R$ the cokernel of $\iota_{c,c+\epsilon}$ is empty whenever $\epsilon>0$ is sufficiently small. 
This cokernel can be (non-canonically) identified with a subspace of $C^{c+\epsilon}$ generated by elements of action precisely equal to $c$. 
The compatible basis can then be taken as the union of the elements from the bases of these cokernels.

Any compatible basis can be made into a ``Barannikov decomposition," also called ``canonical form'', by a basis change which is upper-triangular; we refer to \cite[Lemma 2]{Barannikov} for the details.
\end{proof}

Before arriving at the final description of the type of families of complexes that we want to study, we begin with the following intermediate definition. A family of complexes $C(t)^{b_t}_{a_t}$ defined on an interval $t \in I \subset \R$ is said to be a {\bf continuous family of filtered complexes (with action window)} if it satisfies the following.
\begin{itemize}
\item The underlying (unfiltered) complexes in $I$ are all canonically identified. 
\item There is a fixed basis for the complexes (under the identification in the previous bullet point) which remains compatible for all filtrations in $I.$ 
\item The action $\ell_t$ on each element in the compatible basis, as a function of $t \in I$ is continuous and almost everywhere differentiable.
\item The starting and endpoints of the action window $[a_t,b_t)$ vary continuously with $t \in I.$
\end{itemize}
Note that the \emph{filtered} isomorphism type of the complex may change in a continuous family.
For example, the relative action of two (compatible) basis elements may change signs.

We then define a {\bf piecewise continuous} (PWC) {\bf family of filtered complexes (with action window)} to be a family of complexes parametrized by $\R$, which is continuous on each of the finitely many connected components of $\R \setminus \{t_1,\ldots,t_K\}$ for some $t_1<\ldots<t_K$, while it satisfies the following.
\begin{itemize}
\item The endpoints of the action window $a_t, b_t \in \R$ vary continuously on all of $\R.$
\item The differential $\partial_t$ is strictly action-decreasing in the sense that the action of $\partial_t(x)$ is \emph{strictly} less than that of $x$ for all $t \in \R \setminus \{t_1,\ldots,t_K\}.$
\end{itemize}
In addition, we prescribe the following behavior at the singular moments $t_1<\ldots <t_K.$ For $\epsilon <\min_{j \ne k} |t_j -t_k|,$ let $\{I_0, I_1\} = \{(t_i, t_i+ \epsilon), (t_i - \epsilon, t_i)\}$ denote two intervals near $t_i$ with their assignments unspecified.
We require that precisely one of the following {\bf simple bifurcations} takes place at each $t_i.$
\begin{itemize}
\item {\bf Handle-slide: } The differential of the complex is changed by an action preserving handle-slide, i.e.~with conjugation by an invertible matrix which is upper-triangular with respect to any basis that is compatible with the filtration.
\item {\bf Birth/death: } There is a continuously varying family $\tilde{C}(t)^{b_t}_{a_t}$ of filtered complexes defined for all $t \in (t_i-\epsilon,t_i+\epsilon)$ and a splitting
$$(C(t_i)^{b_{t_i}}_{a_{t_i}},\partial_{t_i})=(\tilde{C}(t_i)^{b_{t_i}}_{a_{t_i}},\tilde{\partial}_{t_i}) \oplus (S,\partial_S)$$
 into a direct sum of filtered complexes (where the total complex is endowed with the induced filtration), with $S=\kk x \oplus \kk y,$ $\ell_{t_i}(x)=\ell_{t_i}(y)$ and $\partial_S(x)=y$, such that the following holds. 
For $t \in I_0,$ $({C}(t)^{b_t}_{a_t}, \partial_t)$ continuously extends to $(C(t_i)^{b_{t_i}}_{a_{t_i}},\partial_{t_i})$ and for $t \in I_1,$ $({C}(t)^{b_t}_{a_t}, \partial_t)$ continuously extends to $(\tilde{C}(t_i)^{b_{t_i}}_{a_{t_i}},\tilde{\partial}_{t_i}).$
\item {\bf Generator enters/exits at the bottom: } 
Possibly after shrinking $\epsilon>0$, there exists an extension of the family of complexes to a family $(C(t)^{b_t}_{a_t-\epsilon},\partial_{t})$ that depends continuously on $t \in (t_i-\epsilon,t_i+\epsilon)$ and for which there is a canonical identification
$$(C(t)^{b_t}_{a_t},\partial_{t}) = (C(t)^{b_t}_{a_t-\epsilon},\partial_{t})/C(t)^{a_t}_{a_t-\epsilon}$$
with the quotient complex. Further, the quotient map has a one-dimensional kernel for $t \in I_0$ while it is an isomorphism for $t \in I_1.$
\item {\bf Generator enters/exits at the top: }
Possibly after shrinking $\epsilon>0$, there exists an extension of the family of complexes to a family $(C(t)^{b_t+\epsilon}_{a_t},\partial_{t})$ for some $\epsilon>0$ that depends continuously on $t \in (t_i-\epsilon,t_i+\epsilon)$ for which there is a canonical identification
$$(C(t)^{b_t}_{a_t},\partial_{t}) \subset (C(t)^{b_t+\epsilon}_{a_t},\partial_{t})$$
of subcomplexes. Further, this inclusion has a one-dimensional cokernel for $t \in I_0$ while it is an isomorphism for $t \in I_1.$
\end{itemize}

In the standard setup of persistent homology, invariance is studied for complexes that are related by action-preserving ``continuation maps." In particular, continuation chain maps induce isomorphisms on homology. Because we allow generators to enter or exit the action window, our induced homology maps might not be isomorphisms, at which point we replace the standard continuation maps with quotients or inclusions of complexes.



\subsection{The barcode}

In Definition \ref{dfn:barcode}, Lemma \ref{lma:barcode} and Lemma \ref{lma:BarcodeRecovery} below, we fix the parameter $t.$ 
The filtration induces an inclusion
$C(t)^{c_0}_{a_t} \hookrightarrow C_*(t)^{c_1}_{a_t}$
of complexes for each $c_0 \le c_1 \le b_t.$
Denote the induced maps on homology by
$$\phi_{c_0,c_1} \colon H(C(t)^{c_0}_{a_t}) \to H(C(t)^{c_1}_{a_t}).$$
The maps satisfy
$$ \phi_{c,c}=\id_{H(C(t)^{c}_{a_t})} \:\: \text{and} \:\: \phi_{c_1,c_2} \circ \phi_{c_0,c_1} = \phi_{c_0,c_2}.$$
\begin{dfn}
\label{dfn:barcode}
A {\bf barcode} is a finite collection of subsets of $\R$, each an interval of the form $[s,e)$, called {\bf bars}. We call $s\in \R$ the bar {\bf starting point}, $e \in \R \cup \{+\infty\}$ the bar {\bf endpoint}, and $e-s > 0$ the bar {\bf length}. A bar $B$ {\bf persists at action level $l \in \R$} if $l \in B$. 
The {\bf barcode of the complex $(C(t)^{b_t}_{a_t},\partial_t)$}, denoted by $\mathscr{B}(C(t)^{b_t}_{a_t},\partial_t),$ is the barcode uniquely characterized by the following properties.
\begin{itemize}
\item The number of bars with starting point $s$ is equal to the dimension of the quotient
$$ \OP{coker}( \phi_{s,s+\epsilon}) = H(C(t)^{s+\epsilon}_{a_t},\partial_t)/\im\,\phi_{s,s+\epsilon}$$
where $\epsilon>0$ is any sufficiently small number.
\item The number of bars with starting point $s$ that persist at action level $l\ge s$ is equal to the dimension of the subspace
$$ [\phi_{s+\epsilon,l+\epsilon}](\OP{coker}( \phi_{s,s+\epsilon})) \subset H(C(t)^{l+\epsilon}_{a_t},\partial_t)/\im\,\phi_{s,l+\epsilon}$$
where $\epsilon>0$ is any sufficiently small number and where the map
$$[\phi_{s+\epsilon,l+\epsilon}] \colon \OP{coker}( \phi_{s,s+\epsilon}) \to H(C(t)^{l+\epsilon}_{a_t},\partial_t)/\im\,\phi_{s,l+\epsilon} $$
is induced by descending $\phi_{s+\epsilon,l+\epsilon}$ to the quotients.
\end{itemize}
\end{dfn}
We begin with some basic facts and characterizations of barcodes that will facilitate their computation, especially in conjunction with the Barannikov decomposition from Lemma \ref{lma:compatiblebasis}. 
\begin{lma}
\label{lma:barproperties}
\begin{enumerate}
\item Two complexes which are related by an action preserving chain isomorphism have the same barcode.
\item The barcode of a direct sum of two filtered complexes is the union of the respective barcodes.
\item A one-dimensional filtered complex spanned by the element $c$ (necessarily a cycle) has a barcode which consists of a single (semi-)infinite bar $[\ell(c),+\infty) \subset \R.$
\item The acyclic complex spanned by two generators $c_0,c_1$ satisfying $\ell(c_0)<\ell(c_1)$ and $\partial(c_1)=c_0$ has a barcode which consists of a single bar $[\ell(c_0),\ell(c_1)) \subset \R.$
\end{enumerate}
\end{lma}
\begin{proof}
This follows from the above definition.
\end{proof}

\begin{lma}
\label{lma:barcode}
\begin{enumerate}
\item Take any choice of basis for the complex $(C(t)^{b_t}_{a_t},\partial_t)$ which is compatible with the filtration, and assume that $\partial_t$ is strictly action decreasing. The total number of bars in $\mathscr{B}(C(t)^{b_t}_{a_t},\partial_t)$ which have either a starting point or endpoint at level $l\in \R$ equals the number of basis elements of precisely that action.
\item The number of bars that persist at action level $l \in \R$ and which have starting points located strictly below the level $c \le l$ is equal to the dimension of the image
$$ \phi_{c,l+\epsilon}(H(C(t)^{c}_{a_t},\partial_t)) \subset H(C(t)^{l+\epsilon}_{a_t},\partial_t)
$$
for $\epsilon>0$ sufficiently small.
\end{enumerate}
\end{lma}
\begin{proof}
We use Lemma \ref{lma:compatiblebasis} to obtain a Barannikov decomposition for the complex $(C(t)^{b_t}_{a_t},\partial_t)$.

(1): The number of basis elements at each action level remains unchanged by the filtration preserving basis change carried out in Lemma \ref{lma:compatiblebasis}. The statement now follows from Lemma \ref{lma:barproperties}: the Barannikov composition exhibits the complex as a direct sum of filtered complexes of dimension at most two, where the two-dimensional complexes are acyclic.

(2): Any Barannikov decomposition for $(C(t)^{b_t}_{a_t},\partial_t)$ induces a Barannikov decomposition on the subcomplexes
$C(t)^{c}_{a_t} \subset C(t)^{l+\epsilon}_{a_t} \subset C(t)^{b_t}_{a_t}$
 by taking appropriate subsets of the basis elements.

The Barannikov decomposition provides a filtered isomorphism between $C(t)^{c}_{a_t}$ and a direct sum of filtered complexes of dimension at most two, where the two-dimensional complexes are acyclic. Since Part (2) is easily seen to hold for any such summand, Parts (1) and (2) of Lemma \ref{lma:barproperties} again show the full claim.
\end{proof}

The following proposition tells us that the properties that are considered in Lemma \ref{lma:barcode} are sufficient to recover the entire barcode.
\begin{lma}
\label{lma:BarcodeRecovery}
The barcode can be recovered from the following data:
\begin{enumerate}
\item the set $c_1<c_2<\ldots<c_k$ of values which is given by the union of action levels of the starting points and endpoints of all bars in the barcode; and
\item the number of bars which persist at level $l_i\in\R$ and start at level $c_j \in \R$, for each $1 \le j \le i \le k$, where we write $l_i=(c_{i+1}+c_i)/2 \in \R$, $i=1,2,\ldots,k-1$ and $l_k=c_k+1.$ 
\end{enumerate}
\end{lma}
\begin{proof}
For each fixed $1 \le j \le k$, the number of bars which start at $c_j$ and end at $c_i$ is equal to the difference between the number of such bars which start at $c_j$ and persist at $l_i$ and the number of bars which start at $c_j$ and persist at $l_{i-1}$. The number of infinite bars that start at $c_j$ is equal to the number of such bars which persist at $l_k$. Obviously this information can be deduced from the data in (1) and (2). Since we can do this analysis for all $j=1,\ldots,k$, it is thus possible to recover the entire barcode.
\end{proof}

\begin{prp}[Barcode Proposition]
\label{prp:Barcode}
Consider a PWC family of filtered complexes with action window, $C(t)^{b_t}_{a_t},$ that undergoes only simple bifurcations.

When the complex undergoes no such bifurcation, the barcode undergoes a continuous change of action levels for its starting and endpoints. 

At the bifurcations the barcode undergoes the following corresponding changes.
\begin{itemize}
\item {\bf Handle-slide: } The barcode is unaffected.
\item {\bf Birth/death: } When two generators $x,y$ undergo a birth/death, then a bar connecting $\ell(x)$ to $\ell(y)$ is added to/removed from the barcode. (The bar is not present at the exact time of the birth/death, but immediately after/before it is visible and of arbitrarily short length.)
\item {\bf Exit below: } A generator slides below the action level $a_t$ at time $t.$ If the uniquely determined \footnote{This bar is uniquely determined near the exit moment by Part (1) of Lemma \ref{lma:barcode} and the definition of an exit below.} bar which starts at the action level of that generator is of finite length, then that bar gets replaced with a bar of infinite length whose starting point is located at the same action level as the endpoint of the original bar. If the bar has infinite length, then it simply disappears from the barcode.
\item {\bf Entry below: } This is the same as a exit below but in backwards time.
\item {\bf Exit above: } A generator slides beyond the action level $b_t$ at time $t.$ There is a uniquely determined bar which either ends or starts at the action level of that generator. In the first case, the bar gets replaced with one that has the same starting point but which is of infinite length. In the second case, when the bar necessarily is infinite, then that bar simply disappears from the barcode.
\item {\bf Entry above: } This is the same as an exit above, but in backwards time.
\end{itemize}
\end{prp}

\begin{proof} 
Assume that the bifurcations occur at the finite subset of times $\{t_1,\ldots,t_K\} \subset \R$. Lemma \ref{lma:BarcodeRecovery} can be used to recover the barcode by the data supplied by Lemma \ref{lma:barcode}. Since the latter data varies continuously for any continuous change of action for a compatible basis, it follows that the barcode depends continuously on $t$ inside each component of $\R \setminus \{t_1,\ldots,t_K\}$.

We can then consider each bifurcation moment separately.

\begin{itemize}
\item {Handle-slide: } The effect of this bifurcation on the barcode is immediate from Part (1) of Lemma \ref{lma:barproperties}, since the bifurcation at a handle-slide moment is an action preserving chain isomorphism.

\item {Birth/death: } The effect of this bifurcation on the barcode can be deduced by Parts (2) and (4) of Lemma \ref{lma:barproperties}, since at the birth moment the complex is changed by taking the direct sum with an acyclic filtered complex of dimension two.
\item {Entry/exit above: } There are two possibilities for the bifurcation: either the filtered complex is changed by the addition/removal of a direct summand of dimension one, or the complex has an acyclic summand of dimension two in which the generator of greater action exits/enters at level $b_{t_i}$. One readily deduces the change of barcode by using Lemma \ref{lma:barproperties} in the two different cases.
\item {Entry/exit below: } Again there are two possibilities for the bifurcation: either the filtered complex is changed by the addition/removal of a direct summand of dimension one, or the complex has an acyclic summand of dimension two in which the generator of less action exits/enters at level $a_{t_i}$. Again one can apply Lemma \ref{lma:barproperties}. 
\end{itemize}
 \end{proof}

\section{Proofs of Theorem \ref{thm:main} and Lemma \ref{lma:sharp}}
\label{sec:main}

First we compare the oscillation and the change in Reeb chord length. Then we study a length-filtered invariance property for the linearized Chekanov--Eliashberg DGA. We apply it to a two-component Legendrian, which includes the two-copy link of a single Legendrian.
The infinite dimensional DGA has its disadvantages. So we construct a partial linearization of the Chekanov--Eliashberg DGA inside a finite action window, which is an associated finite-dimensional complex. The main point is that this linearization can be done even when there is no augmentation of the Chekanov--Eliashberg DGA. Proposition \ref{prp:Bifurcation} uses bifurcation analysis to show that when varying the geometric data, the partial linearizations form a PWC family of complexes with action window as defined in Section \ref{sec:Barcodes}. This allows us to apply the theory barcodes developed in the same section, in particular Proposition \ref{prp:Barcode}, which we then use to prove Theorem \ref{thm:main}. We end with the proof of Lemma \ref{lma:sharp}.

Denote by $\Lambda(t)$ a Legendrian isotopy parametrized by $t.$
Let $H_t$ be the contact Hamiltonian $H_t \colon Y \to \R$ generating an ambient contact isotopy inducing $\Lambda(t),$ and let $X_t$ denote the contact vector field. 

\subsection{Reeb chord length and oscillatory energy}

\label{sec:ReebAction}

The filtration properties depend on the size of the oscillation of the contact Hamiltonian inducing the Legendrian isotopy. The main contact geometric property that we need is the following.
\begin{lma}
\label{lma:main}
A smooth one-parameter family $c(t) \subset (Y,\alpha)$ of Reeb chords with endpoint and starting point $e(t), s(t) \in \Lambda(t)$ on a family of Legendrian submanifolds satisfies
\begin{equation}
\label{eq:contham}
\frac{d}{dt}\ell (c(t)) = \alpha(X_{e(t)}(t)) - \alpha(X_{s(t)}(t)).
\end{equation}
In particular $|\ell (c(0)) - \ell (c(1))| \le \|H_t\|_{\OP{osc}}.$
\end{lma}

\begin{proof}
Cartan's formula gives us
$$\frac{d}{dt}(\phi^t_{\alpha,H_t})^*\,\alpha =(\phi^t_{\alpha,H_t})^*(\iota_{X(t)}d\alpha+d\iota_{X(t)}\alpha).$$
Using this we compute
$$ \frac{d}{dt} \int_{c(t)}\alpha=\frac{d}{dt}\int_{(\phi^t_{\alpha,H_t})^{-1}\circ c(t)}(\phi^t_{\alpha,H_t})^*\,\alpha=\int_{(\phi^t_{\alpha,H_t})^{-1}\circ c(t)}\frac{d}{dt}(\phi^t_{\alpha,H_t})^*\,\alpha=\int_{c(t)} d\iota_{X(t)}\alpha.$$
For the second equality we have use the fact that a Reeb chord $(\phi^t_{\alpha,H_t})^{-1}\circ c(t)$ with endpoints on Legendrian $\Lambda$ are critical points for the functional $\gamma \mapsto \int_\gamma (\phi^t_{\alpha,H_t})^*\,\alpha$ (with $t$ fixed). For the last equality we combine Cartan's formula with the fact that the one-form $\iota_{X(t)}d\alpha$ pulls back to zero on any Reeb chord (by the definition of the Reeb vector field).
\end{proof}

\subsection{Handle-slides and birth/deaths for the DGA}
\label{sec:HSBD}

Lemma \ref{lma:main} uses only elementary calculus and applies to a Legendrian isotopy in any contact manifold $Y.$
Henceforth, however, we need to assume as stated in the introduction, that
$(Y, \xi) = (P\times \R, \ker\{ dz+\lambda\}).$ 

Given the Legendrian isotopy $\Lambda(t)$ and constants $0 \le t_- \le t_+ \le 1,$ denote the stable-tame DGA morphism constructed in \cite{EES07} by
$$\Phi_{t_-, t_+} \colon (\mathcal{A}(\Lambda(t_-)),\partial_{t_-}) \to (\mathcal{A}(\Lambda(t_+)),\partial_{t_+}).$$

A generic Legendrian isotopy has isolated singular moments during which exactly one of the following occurs: a unique rigid index $-1$ disk appears ({\bf handle-slide}); a unique pair of Reeb chords appears/cancels ({\bf birth/death}); or, the relative actions of two Reeb chords changes signs.

 Let here (and elsewhere) $\delta^x_y$ be 0 if $x \ne y,$ and $\delta^x_y$ be some unit in $\kk$ if $x=y.$
Also, let $c^+$ be the chord in $\Lambda(t + \epsilon)$ representing the image of the Reeb chord $c^-$ of $\Lambda(t - \epsilon)$ under the isotopy. 

If the handle-slide disk $u \in \mathcal{M}_A(a, b_1 \cdots b_k)$ exists at time $t$ then by \cite{EES05b, EES07} the induced DGA morphism for $\epsilon>0$ arbitrarily small is
\begin{equation}
\label{eq:HandleSlide}
\Phi_{t - \epsilon, t+\epsilon}(c^-) = c^+ + \delta^{a^-}_{c^-} b_1^+ \cdots b_k^+.
\end{equation}

We extend the definition of {\bf{length}} from Section \ref{sec:Introduction}. For any non-zero element in the algebra $x\in\mathcal{A}(\Lambda),$ let 
$\ell(x) \in [0,+\infty)$
 be the maximum of sums of lengths of Reeb chords in a nonzero word of Reeb chord generators that appears in $x.$ 
We use here that we have a canonical basis of $ \mathcal{A}(\Lambda)$ given by the words of Reeb chords. 
Stokes' Theorem implies
$$ \ell(a^\pm) 
\ge
 \ell(b_1^\pm\cdots b_k^\pm).$$

We review the induced algebraic continuation map at a birth-moment in the proof of \cite[Lemma 2.13]{EES05b} below.
Suppose $a^+,b^+$ are the newly-born pair of points at time $t$ which exists at time $t+\epsilon$ with $|a^+| = |b^+|+1.$
For all sufficiently small $\epsilon>0$, we can assume that the other $(l+k) \ge 0$ chords satisfy
\[
\ell(a_k^\pm) > \ldots > \ell(a_1^\pm) > \ell(a^+) > \ell(b^+)
> \ell(b_l^\pm) > \ldots > \ell(b_1^\pm)
\]
for all $2^{k+l}$ possible assignments of signs $\pm.$ 

Let $S(\mathcal{A}(\Lambda(t - \epsilon)))$ denote the DGA-stabilization of 
$\mathcal{A}(\Lambda(t - \epsilon)).$
Recall from \cite{EES05b} this means we append to $\mathcal{A}(\Lambda(t - \epsilon))$ the (``artificial") generators $a^-$ and $b^-$ with $\partial a^- = b^-, \partial b^-=0.$
The induced DGA-map $\Phi_{t-\epsilon, t+\epsilon}: S(\mathcal{A}(\Lambda(t - \epsilon)))\rightarrow \mathcal{A}(\Lambda(t + \epsilon))$ is defined inductively on the $a_i^-$ as ordered by their action/subscript.
For the base case (the map $\Phi_{t-\epsilon, t+\epsilon}$ restricted to the sub-DGA generated with no $a_i^-$ generators), we define 
\begin{equation}
\label{eq:BirthDeath1}
\Phi_0(c^-) = c^+ + \delta_{c^-}^{b^-} (\partial_{t+\epsilon} a^+ - b^+).
\end{equation}
Again note that any word $w$ appearing in $(\partial_{t+\epsilon} a^+ - b^+)$ satisfies
$$ \ell(b^+) 
\ge
 \ell(w).$$

Define the algebra morphism $f: \mathcal{A}(\Lambda(t + \epsilon)) \rightarrow \mathcal{A}(\Lambda(t + \epsilon))$ on words $w^+$ that contain the letter $b^+$ by replacing the first occurrence of the letter $b^+$ with $a^+$:
$$f(w^+) = \delta_{w^+}^{\alpha^+ b^+ \beta^+} \alpha^+ a^+ \beta^+.
$$
Here $\alpha^+$ is not divisible by $b^+.$ Observe that
$$ \ell(\alpha^+ a^+ \beta^+) -\ell(w^+) = \ell(a^+)-\ell(b^+)>0$$
can be assumed to be arbitrarily small.

Assume $\Phi_{i-1}$ is defined (i.e.~the map $\Phi_{t-\epsilon, t+\epsilon}$ restricted to the sub-DGA which is generated by $b_1^-, \ldots, b_l^-,b^-,a^-,a_1^-, \ldots a_{i-1}^-$).
Then $\Phi_{i} = g_i \circ \Phi_{i-1}$ where 
\begin{equation}
\label{eq:BirthDeath2}
g_i(c^-) = c^+ + \delta_{c^-}^{a_i^-} f \circ \partial_{t + \epsilon}(a_i^+).
\end{equation}

The map $\Phi_{t-\epsilon, t+\epsilon}$ may be viewed as a sequence of {\bf artificial handle-slide maps}:
 each rigid disk in $\mathcal{M}(a_i^+; \alpha^+ b^+ \beta^+)$ contributes to a handle-slide map
 $a_i^- \mapsto a_i^+ + \alpha^+ a^+ \beta^+;$ and 
 each rigid disk in $\mathcal{M}(a^+; x^+_1 \cdots x^+_n)$ with $x^+_1 \cdots x^+_n \ne b^+$ contributes to a handle-slide map
 $b^- \mapsto b^+ + x_1^+ \cdots x_n^+.$ 
 
The above considerations on Reeb chord lengths implies that $\Phi_{t-\epsilon,t+\epsilon}$ can only increase the action by an arbitrarily small amount for sufficiently small $\epsilon.$

\begin{lma}
\label{lma:GenericIsotopy}
Assume $\Lambda(0)$ and $\Lambda(1)$ are generic.
For any $\delta >0$ there exists a contact Hamiltonian $H'_t$ such that the induced isotopy $\Lambda'(t)$ is generic as above, $\Lambda'(i) = \Lambda(i)$ for $i=0,1$
and $\left|\|H_t\|_{\OP{osc}} - \|H'_t\|_{\OP{osc}}\right| < \delta.$
\end{lma}
\begin{proof}

 A $C^\infty$-small perturbation of the Legendrian isotopy provides the needed genericity used in \cite[Proposition 2.9]{EES05b} to prove families of moduli spaces of $J$-holomorphic curves are transversely cut out. Inside a family of one-jet neighborhoods of the Legendrians where the Legendrian is identified with the zero-section (see e.g.~\cite[Theorem 6.2.2]{Geiges}), such a perturbation can be described as the one-jet of a family of $C^\infty$-small real-valued functions defined on the Legendrian. It is now clear that $H'_t$ can be obtained from $H_t$ by the addition of a $C^\infty$-small function on the contact manifold which coincides with the latter family of functions near the Legendrian (appropriately extended from the Legendrian to the entire jet-space neighborhood).
\end{proof}

\subsection{Partial linearizations for the two-component link}
\label{sec:FilteredInvariance}

We show how to linearize the Legendrian contact homology DGA inside an action window, and investigate its invariance properties. For a generic one-parameter family of data, we obtain a PWC family of filtered complexes in the sense of Section \ref{sec:ContFamily}.

Consider a Legendrian isotopy $\Lambda(t)$ of a two-component link. (Each ``component" may itself be disconnected and have interesting topology but we do not consider such sub-components individually.)
There are two types of chords: {\bf pure} chords $\mathcal{Q}_{\OP{pure}} = \mathcal{Q}_{\OP{pure}}(t)$ which start and end on the same component; and, {\bf mixed} chords which run between the two different components.

The sub-algebra $\mathcal{A}_{\OP{pure}}(\Lambda(t))$ freely generated by pure chords $\mathcal{Q}_{\OP{pure}}$ is of course a DGA of its own, given as the free product of the DGAs for the two different components. Henceforth we will only consider an augmentation $\varepsilon$ which vanishes on each mixed chord, i.e.~which is induced by an augmentation of $\mathcal{A}_{\OP{pure}}(\Lambda(t))$. Note that an augmentation of the second type always induces an augmentation of the first type by elementary topological reasons: the differential of a mixed chord must output words in which at least one chord is mixed. For an ordering $\Lambda_0(t) \sqcup \Lambda_1(t) =\Lambda(t)$ of the two components, let $\mathcal{Q}_{\OP{mixed}} = \mathcal{Q}_{\OP{mixed}}(t)$ denote the mixed chords {\em{starting}} at $\Lambda_0(t)$ and {\em{ending}} at $\Lambda_1(t).$
Let $LCC_*^\varepsilon(\Lambda(t))$ the induced {\bf linearized (chain) complex} generated by $\mathcal{Q}_{\OP{mixed}}.$ The (linearized) differential $\partial^\varepsilon$ counts holomorphic disks with a positive puncture at a mixed chord and the augmentation applies to all but one of the negative punctures (and thus the output is thus again a mixed chord of the first type). We refer to \cite{BourgeoisChantraine} for more details.

We let $LCC_*^\varepsilon(\Lambda(t))^b_a$ denote the linearized subcomplex generated by the subset of mixed chords in $\mathcal{Q}_{\OP{mixed}}$ having lengths contained in the interval $[a,b).$ The arguments of \cite{BourgeoisChantraine} which imply that $LCC_*^\varepsilon(\Lambda(t))$ is well-defined combined with a standard filtered chain complex argument and Stokes' Theorem, imply $LCC_*^\varepsilon(\Lambda(t))^b_a$ is well-defined (but of course not necessarily invariant).
We call $[a,b)$ the {\bf action window}. 

We are also interested in the case when the DGA of $\Lambda(t)$ might not have an augmentation, but at the same time, the sub-DGA $\mathcal{A}_{\OP{pure}}^l(\Lambda(t))\subset \mathcal{A}_{\OP{pure}}(\Lambda(t))$ generated by only the pure chords of length less than some fixed number $l \ge 0$ admits an augmentation
$$ \varepsilon \colon \mathcal{A}_{\OP{pure}}^l(\Lambda(t)) \to \kk.$$

Consider the subspace $\mathcal{A}_1^l \subset \mathcal{A}(\Lambda)$ spanned by words of chords of which
\begin{itemize}

\item precisely one is in $\mathcal{Q}_{\OP{mixed}}$ (so this chord starts on $\Lambda_0$ and ends on $\Lambda_1$) whose length lies in the interval $[a,b),$ while
\item the remaining chords are all pure and with lengths each less than $l.$
\end{itemize}
This subspace can naturally be identified with the free $\mathcal{A}_{\OP{pure}}^l$--bimodule generated by the chords $\mathcal{Q}_{\OP{mixed}}$ of lengths in the interval $[a,b).$ This bimodule can be made into a chain complex, which we denote by $LCC_*^{l,\varepsilon}(\Lambda(t))^b_a.$ Since this complex is new to the literature, we describe its (linear) differential below.

 There is an automorphism $\Phi_\epsilon \colon \mathcal{A}_1^l \to \mathcal{A}_1^l$ given as the restriction of the algebra-map that is defined by $c \mapsto c+\varepsilon(c)$ on each generator (which by assumptions on $\varepsilon$ thus fixes the mixed generators). Let $\pi_\varepsilon: \mathcal{A} \rightarrow \mathcal{A}_1^l \subset \mathcal{A}$ be the canonical projection $\mathcal{A} \rightarrow \mathcal{A}_1^l $ induced by our canonical basis, post-composed with $\Phi_\epsilon.$ 
 The linearized differential can then be expressed as the linear part of $\pi_\varepsilon \circ \partial$ restricted to the vector subspace $LCC_*^{l,\varepsilon}(\Lambda(t))^b_a \subset \mathcal{A}_1^l$
 spanned by the mixed chords, which is a map
$$ \partial^\varepsilon \coloneqq (\pi_\varepsilon \circ \partial)_1 \colon LCC^{l,\varepsilon}_*(\Lambda)_a^b \to LCC^{l,\varepsilon}_*(\Lambda)_a^b.$$

\begin{lma}
\label{lma:WellDefinedness}
If $b-a \le l$ then $LCC_*^{l,\varepsilon}(\Lambda(t))^b_a$ is a well-defined complex, i.e.~$(\partial^\varepsilon)^2=0$, which we call the {\bf partially linearized complex (with action window $[a,b)$)}.
\end{lma}

\begin{proof}
The inequality $b-a \le l$ implies that when counting the glued pairs of disks which contribute to $(\partial^\varepsilon)^2,$ the Reeb chord at which the gluing occurs cannot be a one for which the augmentation is not defined. This reduces the $\left[(\partial^\varepsilon)^2=0\right]$-proof to the established case when the augmentation is globally defined.
\end{proof}

We need to set-up a bifurcation analysis to prove a certain form of invariance for the complex $LCC_*^{l,\varepsilon}(\Lambda(t))^b_a$ as the parameter $t$ varies. We start by considering the case of a singular moment $t$ for the bifurcation of the DGA, at which no chord has length equal to $l.$

Choose $\epsilon>0$ sufficiently small so that the chords which do not undergo a birth/death are preserved for all $s \in [t-\epsilon,t+\epsilon].$
In the event of a birth of pair of chords $a^+,b^+$ we suppress the stabilization notation, using $\mathcal{A}(\Lambda(t-\epsilon))$ for $S(\mathcal{A}(\Lambda(t-\epsilon))).$

\begin{lma}
\label{lma:LinearHandleSlide}
Assume that no chord at time $t$ has length equal to either of the values $l,a,b \in \R$ and that there exists an augmentation $\varepsilon \colon \mathcal{A}_{\OP{pure}}^l(\Lambda(t-\epsilon)) \to \kk$ for all sufficiently small $\epsilon>0.$ Let $$\Phi_{t-\epsilon,t+\epsilon} \colon (\mathcal{A}(\Lambda(t-\epsilon)), \partial) \rightarrow (\mathcal{A}(\Lambda(t+\epsilon)), \partial')$$ be either a birth/death or a handle-slide as in (\ref{eq:HandleSlide}). For sufficiently small $\epsilon>0,$ there exists an augmentation 
$$\varepsilon' \colon \mathcal{A}^l_{\OP{pure}}(t+\epsilon) \to \kk$$ 
 for which the map
$$\phi_\varepsilon := \pi_{ \varepsilon'} \circ \Phi_{t-\epsilon,t+\epsilon} \colon LCC_*^{l,\varepsilon}(\Lambda(t-\epsilon))^b_a \to LCC_*^{l,\varepsilon'}(\Lambda(t+\epsilon))^b_a$$
is a composition of action-preserving (linear) chain maps $c \mapsto c+f$ with $\ell(f)<\ell(c)$ (here we include the case $f=0$) together with cancellations of death pairs.
\end{lma}

\begin{proof}
Since $\epsilon$ is small, we can assume the same set of pure chord generators for $\mathcal{A}_{\OP{pure}}^l(\Lambda(t-\epsilon)$
and for $\mathcal{A}_{\OP{pure}}^l(\Lambda(t+\epsilon)),$
as well as the same set of mixed chord generators for $LCC_*^{l,\varepsilon}(\Lambda(t-\epsilon))^b_a$ and for $LCC_*^{l,\varepsilon'}(\Lambda(t+\epsilon))^b_a.$

 The proof reduces to the the three cases when $\Phi_{t-\epsilon,t+\epsilon}$ corresponds to either a single (real or artificial) handle-slide, a single stabilization, or a single destabilization.

\emph{Handle slide:} There exists an inverse $\Phi_{t-\epsilon,t+\epsilon}$ which also is action preserving. The augmentation can thus be taken to be $\varepsilon'=\varepsilon \circ \Phi_{t-\epsilon,t+\epsilon}^{-1}|_{\mathcal{A}^l},$ and we thus get $\varepsilon=\varepsilon' \circ \Phi_{t-\epsilon,t+\epsilon}|_{\mathcal{A}^l}.$ In this case we define
$$\phi_\varepsilon=\pi_{\epsilon'}\circ\Phi_{t-\epsilon,t+\epsilon}|_{LCC_*^{l,\varepsilon}(\Lambda(t-\epsilon))^b_a}.$$

We need to check that $\phi_\varepsilon$ is a chain map of the sought form.

When $\Phi_{t-\epsilon,t+\epsilon}$ is defined by a handle-slide ``disk" (real or artificial) for which all negative pure punctures action less than $l,$ then statement follows from a standard argument; see e.g.~\cite[Section 2.4]{BourgeoisChantraine}. (We can interpret $\phi_\varepsilon$ as a linearization of a DGA morphism in the standard sense.)

Consider the ``disk" with positive puncture $e$ and negative punctures $f_1, \ldots, f_k.$
 The inequality $b-a <l$ and the description of the real and artificial handle-slides in Section \ref{sec:HSBD} imply the following: if $e$ is a pure chord of less length than $l$ or a mixed chord between length $a$ and $b,$ then no $f_i$ can be a pure chord of length greater than $l,$ unless possibly if at least one of the $f_i \in \mathcal{Q}_{\OP{mixed}}$ is a mixed chord with action less than $a.$
In this case, the induced linear map is the identity. That the identity map is a chain map can be seen from the point of view Gromov compactness: the handle-slide disk with a negative mixed chord of action less than $a$ cannot be glued to a disk used to defined the boundary, since the action difference of the mixed chords of such a glued disk is greater than $b-a.$

\emph{Stabilization:} We extend $\varepsilon'$ to vanish on the new pair of generators, and $\phi_\varepsilon$ is simply the canonical inclusion.

\emph{Destabilization:} We let $\varepsilon'$ take the same value as $\varepsilon$ on the remaining generators. In the case when the death involves pure chords, the chain map $\phi_\varepsilon$ is the identity. In the case when the death involves mixed chords, $\phi_\varepsilon$ is simply the corresponding quotient. Note that these are chain map, even though it is possible that $\varepsilon'\circ \Phi_{t-\epsilon,t+\epsilon}\neq \varepsilon.$ (However, since we are only considering the destabilization, as opposed to a death together with its artificial handleslides, this is irrelevant.)
\end{proof}

The main difference between the invariance of the usual linearized complex and the invariance of a sequence of complexes $LCC_*^{l,\varepsilon}(\Lambda(t))_{a_t}^{b_t}$ considered here is that generators of the latter can slide in and out of the action window $[a_t,b_t)$.

In the following we assume that there exists an augmentation
$$ \varepsilon \colon \mathcal{A}_{\OP{pure}}^l(\Lambda(0)) \to \kk$$
for some $l \in \R$ and write
\begin{equation}
\label{eq:l(t)}
 l(t) \coloneqq \int_{0}^{t} \left(\max_{\Lambda(s)} H_s-\min_{\Lambda(s)} H_s\right) ds = \|H_t\|^{0,t}_{\OP{osc}}.
 \end{equation}

\begin{prp}
\label{prp:Bifurcation}

Consider a smooth path of Legendrians $\Lambda(t)$ for which $\Lambda(0)$ and $\Lambda(1)$ are generic. After the perturbation supplied by Lemma \ref{lma:GenericIsotopy} we can moreover assume the following. First, there exist augmentations
$$\varepsilon_t \colon \mathcal{A}_{\OP{pure}}^{l-l(t)}(\Lambda(t)) \to \kk$$
with $\varepsilon_0=\varepsilon$. Second, the induced family of partially linearized complexes
$$LCC_*^{l-l(t),\varepsilon_t}(\Lambda(t))_{a_t}^{b_t}, \:\: b_t-a_t \le l-l(t),$$
with smoothly varying $a_t,b_t$ form a PWC family of filtered complexes with action window (see Section \ref{sec:ContFamily} for the definition). Here the filtration is induced by the Reeb chord length and the compatible basis is given by Reeb chords.
\end{prp}
\begin{proof}
Lemma \ref{lma:LinearHandleSlide} shows that the complexes undergo standard bifurcations, i.e.~handle slides and birth/deaths, except at the following moments:
\begin{itemize}
\item[(M)] There exist mixed chords on $\Lambda(t)$ of length equal to $a_t$ or $b_t$.
\item[(P)] There exist pure chords on $\Lambda(t)$ of length equal to $l-l(t)$;
\end{itemize}
A genericity argument as in the proof of Lemma \ref{lma:GenericIsotopy} allows us to assume that moments (M) and (P) are isolated and generic in the sense that at most one chord is involved at each separate moment in time, and also that they are disjoint from any handle-slide or birth/death occurring in the Chekanov--Eliashberg algebra. To see this, note that via the front projection, neighborhoods of Reeb chords can be modeled by the neighborhood of a critical point of the difference of two smooth functions. Since Morse functions generically have distinct critical values, we can isolate our bifurcation moments.

One can directly verify that the bifurcations from (M) are precisely the entrances and exits as described in Section \ref{sec:ContFamily}.

The bifurcations (P) involving pure chords have the property that no pure chord can \emph{enter} the action window $[0,l-l(t))$ at any $t \in [0,1]$. Indeed, by Lemma \ref{lma:main}, the speed by which the length of any Reeb chord on the Legendrian shrinks is strictly less than $\max_{\Lambda(s)} H_s-\min_{\Lambda(s)} H_s$, which also is precisely the speed by which the action window $[0,l-l(t))$ is shrinking.

We are thus left with the investigation of a moment when a pure chord $c$ exits the action window $[0,l-l(t))$ at some given time $t=t_0$. More precisely, the chord $c$ is of length $\ell(c)<l-l(t)$ when $t<t_0$ while $\ell(c)=l-l(t)$ at time $t=t_0$. In this case there is a canonical inclusion
$$\mathcal{A}^{l-l(t_0)}_{\OP{pure}}(\Lambda(t_0)) \subset \mathcal{A}^{l-l(t_0-\epsilon)}_{\OP{pure}}(\Lambda(t_0-\epsilon))$$
of DGAs for all $\epsilon>0$ sufficiently small, where the sub-DGA is generated by the same generators as the larger DGA except for the chord $c$.
Take $\varepsilon_{t_0}$ to be the pull-back of $\varepsilon_{t_0-\epsilon}$ under any of the above inclusion. 
We claim that the corresponding partial linearizations are canonically isomorphic for these two augmentations. The crucial property used here is that, as shown by an action consideration, if $x$ and $y$ are mixed chords for which $\ell(x)-\ell(y) \in [a_{t_0},b_{t_0})$, then any word that contains $y$ and appears in the DGA differential $\partial(x)$ at $t=t_0$ does not contain any chord of length $l-l(t_0)$. (The differential $\partial$ is strictly action decreasing.) Since the words that contribute to the linearized differentials thus cannot consist of a word that contains the letter $c$ at either moment $t=t_0,t_0-\epsilon$, the linearized differentials on the mixed chords agree as sought.
\end{proof}

\subsection{Completing the proof of Theorem \ref{thm:main}}

We use the notation introduced in Section \ref{sec:Introduction} before the theorem statement.
In particular, $\Lambda$ is some Legendrian, not the link from the previous subsection.

With an arbitrarily small change in the Legendrian and the contact Hamiltonian, we can assume all moduli spaces below a pre-determined index are transversely cut out \cite{EES07}. In particular Lemma \ref{lma:GenericIsotopy} and \cite[Theorem 3.6]{Ekholm:Duality} (used below) hold.

Let $\bar{\Lambda} = \Lambda \sqcup \Lambda'$ be the 2-copy of $\Lambda$ where the second copy $ \Lambda'$ is translated in the positive Reeb direction by $N \gg \max_c \ell(c).$ Here the max is taken over all pure Reeb chords.
We perturb $\Lambda'$ by a small Morse function with $C^2$-norm bounded by 
$\epsilon \ll \min_c \ell(c).$ Here the min is taken over all Reeb chords.
For each pure chord $c$ of $\Lambda$ there are two mixed chords $p^+_c,p^-_c$ such that the projections
 $P \times \R \rightarrow P$ of $c, p^+_c, p^-_c$ are within $\epsilon,$ and such that 
$|\ell(p^\pm_c) - (N \pm \ell(c))| < \epsilon.$
The remaining mixed chords are called the Morse chords, which we denote by $x.$ They correspond to the critical points of the Morse function and satisfy $|\ell(x)-N| < \epsilon.$

\begin{prp}
\label{prp:ModuliSpaceLift}
For a suitable choice of almost complex structure and sufficiently small Morse perturbation of $\overline{\Lambda},$ the following can be said about the punctured pseudoholomorphic disks on the two-copy with one positive puncture and precisely two punctures at mixed chords. (It follows that the positive puncture must be mixed.)
\begin{enumerate}
\item When both mixed chords are Morse chords, any rigid disk corresponds to a unique rigid gradient flow-line for the Morse function considered.
\item Let $x$ be a Morse chord of index $k.$ 
For every rigid disk with $x$ as the unique positive (resp. unique mixed negative) puncture, the other mixed puncture is a negative puncture at some $p_c^-$ (resp. positive puncture at some $p_c^+$). Moreover, the existence of such a rigid disk implies that the moduli space $\mathcal{M}^{n-k}(c)$ (resp. $\mathcal{M}^{k}(c)$) is non-empty.

\item Assume that the Morse function has a unique max and min on each connected component. The disks having the max chord (resp. min chord) as a positive (resp. negative puncture) are small triangles that are in a one-to-one correspondence with gradient flow-lines that connect the max (resp. min) and an endpoint of a Reeb chord on $\Lambda.$ Furthermore, the max is a cycle and the min is a cocycle for the linearized differential, if the \emph{same} augmentation is used on both components $\Lambda$ and $\Lambda'$. 

 \item Assume that the Morse function has a unique max and min on each connected component, both of which moreover are located sufficiently close to a given generically chosen point $\pt_i \in \Lambda,$ $i \in \pi_0(\Lambda).$ Any disk with the negative puncture at the max chord (resp. positive puncture at the min chord) corresponds to a moduli space $\mathcal{M}^{n}(c)$ for which the boundary of the disc passes through the ray $\R \times \pt_i.$ 
\end{enumerate}
\end{prp}

\begin{proof}

(1): This is Part (4) of \cite[Theorem 3.6]{Ekholm:Duality}.

(2): This is Part (3) of \cite[Theorem 3.6]{Ekholm:Duality}, which identifies the disks on the two-copy with appropriate generalized pseudoholomorphic disks on one copy, together with the dimension formula \cite[(3.11)]{Ekholm:Duality} for the generalized pseudoholomorphic disks.

(3): The first part is Theorem 5.5 and Remark 5.6 of \cite{Ekholm:Duality}. The second part follows analogously, see e.g.~Lemma 4.21 of \cite{CDRGG18}.

(4): This follows as a special case in the proof of (2). Note that the moduli spaces $\mathcal{M}^n(c)$ defined for different choices of point constraints $\pt'_i$ all can be canonically identified, if we assume that the moduli space defined with point constraint at $\pt_i \in \Lambda$ is transversely cut out and that $\pt'_i$ is sufficiently close to $\pt_i.$

\end{proof} 

 We consider the linearized Legendrian contact homology complex as defined in Section \ref{sec:FilteredInvariance}, determined by the choice of ordering $\Lambda_0=\Lambda$ and $\Lambda_1=\Lambda'$ of the components of $\bar{\Lambda}.$ 
 Define $\bar{\Lambda}(t)$ by fixing $\Lambda$ and apply the $N$-vertical-shift of the isotopy to $\Lambda'.$ In this manner we may assume that the component $\Lambda$ is contained outside of the support of $H_t$.

In the following we endow $\mathcal{A}^l(\overline{L})$ with the same augmentation on both components, and suppress it from notation. Proposition \ref{prp:ModuliSpaceLift} (1) implies that $LCC_*^{l}(\bar{\Lambda})^{N+\epsilon}_{N-\epsilon}$ is quasi-isomorphic to the Morse homology of $\Lambda.$
These homology classes need not survive in $LCC_*^{l}(\bar{\Lambda})$ when increasing the action range.
(For example, if $\Lambda$ has an augmentation and can be pushed off its Reeb-flow image, then $LCC_*^\infty(\bar{\Lambda})^{+\infty}_{-\infty}$ is acyclic.)
So if $x$ (a linear combination of Morse chords) represents a degree $k$ homology generator of this Morse subcomplex, its action value
is a starting or ending point of a bar (under an association from Part (1) of Lemma \ref{lma:barcode}) in the barcode for $LCC_*^l(\bar{\Lambda})^b_a$ for any $b-a <l,$ $a+\epsilon \le N \le b- \epsilon$.

Suppose $\epsilon \ll \|H_t\|_{\OP{osc}} < \min(l, \sigma_k)-\epsilon.$ Recall $l(t)$ from equation (\ref{eq:l(t)}).
 Let $x_1 = x, x_2, \ldots, x_n$ denote representatives of the homology classes of degree $k'$ such that $\sigma_{k'} \ge \sigma_k.$ We wish to show that $\bar{\Lambda}(1)$ has at least $n$ mixed Reeb chords. 
For each $x_i$ 
we consider two one-parameter families of barcodes
$$\mathscr{B}^{i,\pm}_t = \mathscr{B}\left( LCC_{*}^{l - l(t)}(\bar{\Lambda}(t))^{b^{i,\pm}_t}_{a^{i,\pm}_t}\right),$$
where both action windows are shrinking at a rate of $l'(t)$ and their small overlap contains the start or endpoint of a bar ``associated to $x_i$"; the precise recipe for how to construct $a^{i,\pm}_t$ and $b^{i,\pm}_t$, thereby giving rise to the associated pair of families of barcodes, will be described below. 

First we make a (generic) assumption that the critical values $\ell(x_i)$ of all Morse critical points are distinct.
Next, we require that
\begin{align}
\notag [a^{i,-}_0, b^{i,-}_0) & = [\ell(x_i) - \min(\sigma_k, l) +2\epsilon, \ell(x_i)+ \epsilon), \\ 
\notag
[a^{i,+}_0, b^{i,+}_0) & = [\ell(x_i) - \epsilon, \ell(x_i)+ \min(\sigma_k, l) -2\epsilon),
\end{align}
 
which uniquely determines $a^{i,\pm}_0$ and $b^{i,\pm}_0$. 
Note that we fix the action window size based to $k,$ independent of $i.$
When extending these parameters to $a^{i,\pm}_t$ and $b^{i,\pm}_t$ we require
\begin{align}
\label{eq:BarPoints1}
& b^{i,\pm}_t - a^{i,\pm}_t = b^{i,\pm}_0 - a^{i,\pm}_0- l(t) < l - l(t),\\
\label{eq:BarPoints2}
&b^{i,-}_t - a^{i,+}_t=2\epsilon.
\end{align}
We uniquely determine $a^{i,\pm}_t$ and $b^{i,\pm}_t$ once we specify $\ell_i(t)$ and require
\begin{equation}
\label{eq:BarPoints3}
\ell_i(t) = \frac{b^{i,-}_t + a^{i,+}_t}{2}.
\end{equation}
Note that $\ell_i(0)=\ell(x_i)$.

The condition \eqref{eq:BarPoints1} enables us to use Proposition \ref{prp:Bifurcation}, which implies the partially linearized complex can be assumed to constitute a PWC family of filtered complexes in the sense of Section \ref{sec:ContFamily}. This in turn allows us to apply Proposition \ref{prp:Barcode} to $\mathscr{B}^{i,\pm}_t.$ Because we have not yet specified $\ell_i(t),$ the barcodes $\mathscr{B}^{i,\pm}_t$ still remain unspecified for $t>0.$


By the genericity argument in the proof of Lemma \ref{lma:GenericIsotopy} we may achieve the following.
\begin{itemize}
\item The lengths of the Reeb chords are all distinct during the isotopy, except for a finite set of ``singular" times $t_j$ when precisely two lengths coincide.
\item At each such $t_j$, the $t$-derivatives of these two Reeb chord lengths are distinct, and $b^{i,\pm}_{t_j}, a^{i,\pm}_{t_j}$ are distinct from all Reeb chord lengths.
\end{itemize}
In particular, it follows that the starting points or end points of the bars in $\mathscr{B}^{i,\pm}_t$ do not intersect for any $t \neq t_j$. 
In the components of the complement of this discrete set $\bigcup_j\{ t_j\}$ the individual bars can thus be uniquely singled out by the location of their starting points, which depend continuously on $t$ away from exits/entrances. Thus, Lemma \ref{lma:barcode} (1) implies that away from $\bigcup_j\{ t_j\}$ there are canonically determined Reeb chords (the compatible basis elements that correspond to the starting points) associated to each bar.

\begin{lma}
\label{lma:InfiniteBars}
Consider an interval $[T_0,T_1)$ disjoint from the singular times $\bigcup_j \{t_j\}.$
Assume that $b^{i,\pm}_t, a^{i,\pm}_t$ and $\ell_i(t)$ have been defined for all $t \in [T_0,T_1]$ satisfying Equations \eqref{eq:BarPoints1}--\eqref{eq:BarPoints3}. 
Moreover, assume $\ell_i(t)$ coincides with the length of a Reeb chord for all $t \in [T_0,T_1].$ 
If $\ell_i(T_0)$ is the starting point of an infinite bar in each of the two $\mathscr{B}^{i,\pm}_{T_0},$ then for all $t \in [T_0,T_1],$ $\ell_i(t)$ is the starting point of an infinite bar in each of the two $\mathscr{B}^{i,\pm}_{t}.$ 
\end{lma}

\begin{proof}
 There is an infinite bar with starting point at $\ell(T_0)$ in each of the two barcodes $\mathscr{B}^{i,\pm}_{T_0}$ by assumption. We need to show that they remain infinite for all $t \in [T_0,T_1]$. By Proposition \ref{prp:Barcode} an infinite bar can only become finite if an endpoint suddenly appears above (entry above), or if the bar switches behavior to a finite bar with endpoint at $\ell_i(t)$ (entry below). In either case, we show that these bifurcations are not possible by a consideration of the needed Hamiltonian oscillation.

 Suppose $\alpha_t < \beta_t$ denote the lengths of two arbitrary mixed chords of $\bar{\Lambda}(t).$ Lemma \ref{lma:main} and the fact that $H_t$ vanishes on $\Lambda_0(t) = \Lambda$ (which contains the Reeb chords' start points) then imply 
$$-\left|\frac{d}{dt}\alpha_t\right| \ge \min_{\Lambda_1(t)} H_t \ge \min_{\bar{\Lambda}(t)} H_t, \quad 
\left|\frac{d}{dt} \beta_t\right| \le \max_{\Lambda_1(t)} H_t \le \max_{\bar{\Lambda}(t)} H_t \longrightarrow
\frac{d}{dt} (\beta_t - \alpha_t) \ge -\frac{d}{dt}l(t).$$ Using Lemma \ref{lma:barcode} Part (1), we consider the two cases when precisely one of $\alpha_t$ or $\beta_t$ is equal to $\ell_i(t)$, while the other is the length of the hypothetical chord which enters from either above or below. Since the size of the action window changes at a rate $\frac{d}{dt} (b^{i,\pm}_t - a^{i,\pm}_t) = -\frac{d}{dt} l(t),$ the above bound combined with the location of $\ell_i(t)$ within the two action windows $[a^{i,\pm}_t, b^{i,\pm}_t),$ proves that for either action window, a generator cannot enter to make finite the infinite bar starting at $\ell_i(t).$ 

\end{proof}

We claim that Lemma \ref{lma:InfiniteBars} implies that the chord of length $\ell_i(T_0)$ at $t=T_0$ cannot undergo a death at any time $T_2 \in [T_0,t_j]$, where $t_j$ is minimal singular time such that $t_j>T_0$.
Suppose the chord dies at $T_2 \in [T_0,t_j].$ 
Extend $\ell_i(t)$ to $[T_1, T_2]$ (if $T_2 > T_1$) by letting $\ell_i(t)$ be the length of (the deformations of) this Reeb chord for $t \in [T_1, T_2].$
Extend $b^{i,\pm}_t$ and $a^{i,\pm}_t$ for $t \in [T_1, T_2]$ using equations (\ref{eq:BarPoints1})--(\ref{eq:BarPoints3}).
Lemma \ref{lma:InfiniteBars} (replacing $[T_0,T_1]$ by $[T_0,T_2]$) implies the chord is the starting point of an infinite (in particular non-zero length) bar at $T_2,$ and therefore cannot die.
So the chord survives to $t_j.$
Extend $\ell_i(t), b^{i,\pm}_t, a^{i,\pm}_t$ to $[T_1, t_j]$ in the identical manner as it was extended to $[T_1,T_2].$
Lemma \ref{lma:InfiniteBars} (replacing $[T_0,T_1]$ by $[T_0,t_j]$)
implies the induced pair $\mathscr{B}^{i,\pm}_t$ of families of barcodes has infinite bars with starting points at $\ell_i(t)$ for all $t \in [T_0,t_j]$. 

We next extend $\ell_i(t)$ continuously through (past) $t_j$ as follows. By the above 
 $\mathscr{B}^{i,+}_{t_j}$ has one or two infinite bars with starting point $\ell_i(t_j)$. 
 (For sufficiently small $\epsilon,$ both bars in $\mathscr{B}^{i,-}_{t_j}$ starting at $\ell_i(t_j)$ are infinite by equations (\ref{eq:BarPoints2}) and (\ref{eq:BarPoints3}).) 
We treat separately these two cases.

\emph{Case (i) (one infinite bar):} There are two chords of the same length at $t_j$ so there are two possibilities for a continuous extension of $\ell_i(t)$ to $t>t_j.$ One is $C^1$ and the other not. Proposition \ref{prp:Barcode} together with the assumption that no entrance/exit can happen near $t_j$ implies there is precisely one such extension for which the bar with starting point at $\ell_i(t)$ remains infinite in $\mathscr{B}^{i,+}_{t_j+\delta}$ for all sufficiently small $\delta \ge 0.$
Choose this infinite extension for $\ell_i(t)$.

\emph{Case (ii) (two infinite bars):} We require that $\ell_i(t)$ is the unique $C^1$-extension over $t_j.$ This is well-defined since the two Reeb chords' lengths have distinct $t$-derivatives at $t_j.$ Proposition \ref{prp:Barcode} together with the assumption that no entrance/exit can happen near $t_j$ implies that $\mathscr{B}^{i,+}_{t_j+\delta}$ contains an infinite bar that starts at $\ell_i(t+\delta)$ for all sufficiently small $\delta \ge 0$.



By induction we can then construct $\ell_i(t)$ for all $t \in [0,1]$. Once we show that (generically) $\ell_i(1) \ne \ell_\iota(1)$ for all $i \ne \iota,$ the number of $x_i$ is a lower bound for the number of mixed chords of $\bar{\Lambda}(1).$ This implies the theorem. 

Suppose $\ell_i(1) = \ell_\iota(1).$ 
Then for some $t_j$, $\ell_i(t_j)=\ell_\iota(t_j)$ and $\ell_i(t_j + \delta) = \ell_\iota(t_j + \delta),$ while $\ell_i(t_j - \delta) \ne \ell_\iota(t_j - \delta)$ holds for all sufficiently small $\delta>0$.

We first claim that the barcode $\mathscr{B}^{i,+}_{t_j}= \mathscr{B}^{\iota,+}_{t_j}$ has two infinite bars with starting point $\ell_i(t_j).$ This is a consequence of the stability of the finite bars as shown in Proposition \ref{prp:Barcode} together with the fact that $\mathscr{B}^{i,+}_{t_j-\delta}$ and $\mathscr{B}^{\iota,+}_{t_j-\delta}$ have infinite bars which start at the two distinct levels $\ell_i(t_j - \delta)$ and $\ell_\iota(t_j - \delta)$, respectively. We now reach a contradiction with our continuous extension convention, by which the bars for $t_j-\delta$ in $\mathscr{B}^{i,+}_{t_j-\delta}$ and $\mathscr{B}^{\iota,+}_{t_j-\delta}$ have distinct $C^1$-extensions through $t_j$.

\subsection{Proof of Lemma \ref{lma:sharp}}
\label{pf:sharp}

Consider the contact isotopy
$$(\mathbf{x},\mathbf{y},z) \mapsto (\mathbf{x},(1-t)\mathbf{y},(1-t)z), \:\: t \in [0,1).$$
Its contact Hamiltonian is equal to $H_t(\mathbf{x},\mathbf{y},z)=-\frac{z}{1-t}$. Using this contact isotopy we can thus rescale the subset $\{|z|\le a/2\}$ to its image of the time-$s$ map for any $0 \le s <1.$ We compute the total oscillation for the isotopy $t \in [0,s],$ $0\le s<1$ restricted to the image of the subset $\{|z|\le a/2\}$ to be equal to
$$\|H_t\|_{\OP{osc}}(\{|z|\le a/2\}) =\int_0^s \left(\frac{(1-t)(a/2)}{1-t}-\frac{(1-t)(-a/2)}{1-t}\right)dt =sa.$$
In other words, since $\Lambda \subset \{|z| \le a/2\}$ by assumption, taking $1-s>0$ to be sufficiently small, we may move $\Lambda$ into an arbitrarily small neighborhood of $\{\mathbf{x} \in D^{ n}, \mathbf{y}=0, z=0\}$ by a contact isotopy of total oscillation equal to $sa <a.$
For any $\delta>0,$ any sufficiently thin neighborhood of $\{\mathbf{x} \in D^{ n}, \mathbf{y}=0, z=0\}$ is displaceable from the subset $D^n \times \R^n \times \R$ by a Hamiltonian of oscillation at most $\delta.$ (Take an ordinary smooth isotopy of $\R^n$ which displaces $D^n \subset \R^n$ and lift it to the jet-space $J^1\R^n=\R^n \times \R^n \times \R$; the corresponding contact Hamiltonian vanishes along the entire zero section and is thus very small near it.) Since $\Lambda \subset D^n \times \R^n \times \R$, the composition creates the sought displacement. In other words, we have shown that $\Lambda$ is displaceable with a total oscillation $a+\epsilon.$ Appropriate cut-offs using smooth bump functions can then be used to make the contact Hamiltonians compactly supported.
\qed

\section{Proof of Theorem \ref{thm:NonSqueeze}}
\label{sec:NonSqueeze}
 Recall that two Legendrians are said to be formally isotopic if they are connected by a smooth isotopy covered by a Lagrangian bundle monomorphism, called the formal tangent map, such that there exists a homotopy relative endpoints from this formal tangent map to the canonical inclusion of the tangent planes (i.e.~the one induced by the smooth family of embeddings) \cite[Definition 1.1]{Murphy:Loose}.

\begin{lma}
\label{lma:FormalIsoRelPt}
Assume that two Legendrian spheres $\Lambda,\Lambda' \subset \R^{2n+1},$ $n \ge 2,$ agree in a neighborhood of a point, and that they are formally Legendrian isotopic. Then there exists a formal isotopy which is fixed in a possibly smaller neighborhood of the same point, in the sense that the underlying smooth isotopy, the Lagrangian frames, as well as the homotopy from the formal tangent map to the inclusion of the tangent planes, all can be taken to be constant there.
\end{lma}
\begin{proof}
One can readily construct an ambient contact isotopy which deforms the initial formal isotopy to one whose underlying smooth isotopy fixes the given point. (Start by smoothly contracting the loop traced out by the point, then use the fact that any family of tangent vectors can be extended to a family of contact vector fields, finally we can now construct a contact isotopy which contracts the loop.) Note that contact isotopies act naturally on formal isotopies by post-composition, and our formal Legendrian isotopy thus may be assumed to fix the point (but not necessarily its neighborhood or even its tangent plane).

After a subsequent deformation by-hand near this point or, alternatively, an application of the one-parametric $h$-principle for open Legendrian embeddings, we may in addition assume that the smooth isotopy near the given point stays Legendrian throughout the isotopy, and that the homotopy from the formal tangent map to the inclusion of the tangent planes is trivial at the point. (We can e.g.~make the embeddings coincide with the Legendrian planes prescribed by the formal bundle morphism near the point.)

What remains is to make the isotopy fix a whole neighborhood of the point. Unfortunately, since $\pi_1(U(n))=\Z$ is non-trivial, the loop of Legendrian tangent frames at the given point is not automatically contractible. However, a final application of a suitable $S^1$-family of contactomorphisms of $\R^{2n+1}$ obtained as lifts of symplectomorphisms from $ U(n) \subset \OP{Symp}(\R^{2n},\omega_0)$, the latter frame over the point may be assumed to give rise to a contractible loop in $U(n).$ Again, since these are contact transformations, they act naturally on formal Legendrian isotopies by post-composition.
\end{proof}

\begin{lma}
\label{lma:cuspsum}
Let $\Lambda$ be an arbitrary Legendrian embedding of dimension $n \ge 2,$ and let $\Lambda'$ be a Legendrian formally isotopic to $\Lambda_{\OP{St}}.$ 
The cusp connected sum between $\Lambda$ and $\Lambda'$ is formally Legendrian isotopic to $\Lambda.$
\end{lma}
\begin{proof}

Cusp connected sum is an operation supported in a neighborhood of an isotropic arc $\gamma$ with endpoints on the two different Legendrians, where the two Legendrians are separated by e.g.~the hyperplane $\{x_1=0\}$; see e.g.~\cite{Dimitroglou:Ambient}. In the dimensions under consideration this operation is well-defined by \cite[Proposition 4.9]{Dimitroglou:Ambient}, and does not depend on the choices made.

First note that the cusp connected sum of $\Lambda$ with the standard Legendrian sphere $\Lambda_{\OP{Std}}$ is Legendrian isotopic to $\Lambda.$ The Legendrian isotopy is easy to construct explicitly if the representative of the standard Legendrian is chosen to be the flying saucer (in the front projection) as shown in Figure \ref{fig:standardsphere}.

We may assume that $\Lambda'$ coincides with the representative of the standard sphere in a neighborhood of the isotropic arc $\gamma$ along which the surgery is performed. The formal isotopy from $\Lambda'$ to $\Lambda_{\OP{St}}$ may further be assumed to have support that is disjoint from the neighborhood of the union $\Lambda \cup \gamma$ by Lemma \ref{lma:FormalIsoRelPt}, together with a general position argument (for the interior of the arc).
\end{proof}

\begin{lma}
\label{lma:looseposition}
Any loose Legendrian $\Lambda \subset \R^{2n+1},$ $n \ge 2,$ is Legendrian isotopic to a representative which satisfies the following property for an arbitrary choice of numbers $A>\delta>0$:
\begin{itemize}
\item There exists a Legendrian fiber $F=\{\mathbf{x}=x_0, z=z_0\}$ for which there are precisely two Reeb chords with one endpoint on $\Lambda$ and one endpoint on the fiber, both which moreover are transverse; 
\item the two Reeb chords between $\Lambda$ and $F$ both start on $\Lambda,$ and their length difference is greater than equal to $A>0$; and
\item the Legendrian $\Lambda$ can be displaced from the fiber by a contact isotopy of oscillation less than $ \delta>0.$
\end{itemize}
(See Figure \ref{fig:loosesphere} for an example.)
\end{lma}

\begin{figure}[htp]
\vspace{5mm}
\centering
\labellist
\pinlabel $z$ at 12 170
\pinlabel $x_1$ at -3 -5
\pinlabel $x_2$ at 129 11
\pinlabel $\Lambda$ at 117 140
\pinlabel $\color{blue}F$ at 59 164
\pinlabel $\color{red}a$ at 53 152
\pinlabel $\color{red}b$ at 61 153
\endlabellist
\includegraphics[scale=1.5]{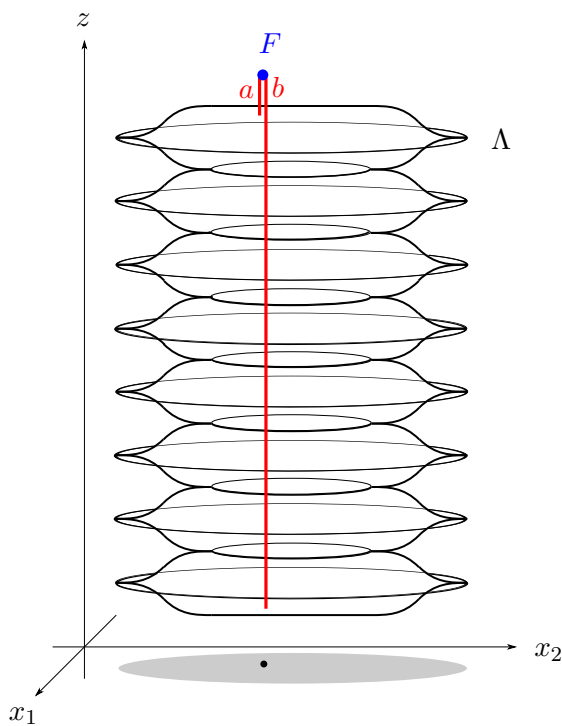}
\vspace{3mm}
\caption{Introducing many zig-zags we may assume that the difference of length between the chords $a$ and $b$ from $\Lambda$ to the fiber $F$ becomes arbitrarily large, while the Lagrangian projection of $\Lambda$ is still contained inside a fixed subset of $\R^4.$}
\label{fig:loosesphere}
\end{figure}

\begin{proof}
In view of Lemma \ref{lma:cuspsum} and Murphy's $h$-principle for loose Legendrians \cite{Murphy:Loose} it suffices to construct a Legendrian sphere in the formal Legendrian isotopy class of the standard sphere that satisfies the properties in the statement. Indeed, it is then a simple matter of taking a cusp connected sum with $\Lambda$ and that sphere.

We begin by constructing the sphere $\Lambda$ of dimension $n=2$ that satisfies the assumptions, and which is formally isotopic to $\Lambda_{\OP{St}}.$ In this dimension there is a unique formal isotopy class of Legendrian spheres \cite{Murphy:Loose}. Considering loose spheres as depicted in Figure \ref{fig:loosesphere}, with sufficiently many zig-zags, one can readily produce sought examples. 

Increasing the number of zig-zags allows us to increase the $z$-coordinate while keeping $y_1,y_2$ small. By Lemma \ref{lma:sharp} the displacing Hamiltonian can then be made small. Thus, a high number of zig-zags makes more optimal the constants $A>\delta>0.$ 

 The construction of the spheres in higher dimensions $\R^{2n+1},$ $n>2,$ can be done by induction. Assume that we have produced the sought embedded sphere $\Lambda^{n-1}$ in dimension $n-1$.

Consider the standard Legendrian $n$-sphere $\Lambda_{\OP{St}}^n,$ which we assume to have its cusp-edge contained above the unit sphere in $\R^n=\{(x_1,\ldots,x_n)\}.$ We then perform a stabilization over a closed domain $U \subset B^n \subset \R^n$ with smooth boundary and Euler characteristic $\chi(U)=0.$ The resulting Legendrian is loose and formally isotopic to the standard sphere; see \cite[Lemma 2.2]{Ekholm:Constructing}. We can find such a domain $U$ which moreover is of the form $[-\epsilon,\epsilon] \times U^{n-1}$ near the hyperplane $\{x_1=0\},$ where $U^{n-1}$ again has vanishing Euler characteristic.

We can assume the stabilization of $\Lambda_{\OP{St}}^n$ intersects the hypersurface
$$\{x_1=0=y_1\} = \{(0,0)\} \times \R^{2n-1} \subset \R^{2n+1}$$
in an $(n-1)$-dimensional Legendrian sphere, which again is loose and formally isotopic to the standard sphere; this intersection is itself the stabilization of $\Lambda^{n-1}_{\OP{St}}$ by $U^{n-1}.$ After a suitable Legendrian isotopy in $\R^{2n-1}$ lifted to $\R^{2n+1}$ we have thus managed to construct a loose Legendrian $n$-sphere in the formal isotopy class of $\Lambda_{\OP{St}}^n$ which coincides with the cylinder
$$[-\epsilon/2,\epsilon/2] \times \{0\} \times \Lambda^{n-1} \subset \{(x_1,y_1)\} \times \R^{2n-1}=\R^{2n+1}$$
in some neighborhood of $\{x_1=0\}.$ The sought fiber $F$ can be found in the same region.
\end{proof}

We now prove Theorem \ref{thm:NonSqueeze} when $n \ge 2$; note that a stabilized Legendrian is loose in these dimension. Given that the statement is shown in higher dimension, we obtain the statement for knots as follows. The front-spinning construction \cite{NonIsoLeg} applied to a stabilized Legendrian knot in $\R^3$ produces a loose Legendrian torus inside $\R^5.$ A potential example of a knot that is squeezed into a neighborhood of a stabilized can thus be spun to an example of an augmentable torus that can be squeezed into a neighborhood of a loose torus. To that end, note that the front spinning preserves Legendrian isotopy classes as well as augmentability \cite{IsotopiesKnots}.

\begin{rmk}
\label{rmk:Lenny}
Lenny Ng pointed out to us that the dimension 1 case can also be proved with rulings. Consider the local picture of an odd-covering of the stabilization. Any attempt at pairing the sheets, necessary for the (local) construction of the ruling, leaves one copy and its two stabilization cusps unpaired. A 1-dimensional knot with an augmentation must have a ruling \cite{FuchsIshkanov, Sabloff}.
\end{rmk}

Denote by $\Lambda_{\OP{loose}} \subset \R^{2n+1},$ $n \ge 2,$ the arbitrary stabilized, hence loose, Legendrian from Theorem \ref{thm:NonSqueeze}, and assume there exists a Legendrian $\Lambda$ which admits an augmentation and can be squeezed into a standard contact neighborhood of $\Lambda_{\OP{loose}}.$

Place the loose Legendrian in the position satisfying the conclusion Lemma \ref{lma:looseposition}. By assumption, we can isotope $\Lambda$ into a standard neighborhood of 
$\Lambda_{\OP{loose}}$ contactomorphic to a neighborhood of the zero section inside $J^1\Lambda_{\OP{loose}}$ \cite[Theorem 6.2.2]{Geiges}.

By a fiber-wise rescaling of $\Lambda$ inside this jet-bundle, together with a general position argument, we can then assume that all mixed Reeb chords between the fiber $F$ (see Lemma \ref{lma:looseposition}) and $\Lambda$ all end on $F$ and start on $\Lambda$, and come of two types: an odd number of Reeb chords of action roughly equal to length $l>0$ and an odd number of Reeb chords of length roughly equal to $A+l,$ where $A \gg 0$ is arbitrarily large. Here we use the assumption that the degree of the bundle projection $$\Lambda \subset J^1\Lambda_{\OP{loose}} \to \Lambda_{\OP{loose}}$$ is of odd degree to infer that both clusters of Reeb chords are odd.

We claim that there must be a bar in the complex $LCC_*^\infty(\Lambda \cup F)^{+\infty}_{-\infty}$ of length at least equal to $A.$ (We use that $\Lambda$ and $F$ have augmentations to be able to set $l = \infty.$) Indeed, the subcomplex consisting of chords of length strictly less than $A$ is odd-dimensional and hence not acyclic, one can then apply Lemma \ref{lma:barcode} to infer the existence of the long bar. Since $\Lambda$ is displaceable from the fiber by a contact isotopy of some fixed small oscillation $\delta>0$ by Lemma \ref{lma:looseposition}, we now arrive at a contradiction by the following argument: Proposition \ref{prp:Bifurcation} shows that the analysis of the bifurcation of barcodes from Proposition \ref{prp:Barcode} can be applied to the family of linearised Legendrian contact homology complexes induced by the displacement. However, using Lemma \ref{lma:main}, the long bars found above cannot disappear in such a family induced by a Hamiltonian of small oscillation.

\section{Proof of Theorem \ref{thm:DisplaceLoose}}

Consider a smooth isotopy $\Sigma_t \subset (X^4,\omega)$ of a compact symplectic embedded surface inside a symplectic four-manifold, where $\partial \Sigma_t \neq \emptyset.$ Assume that $\Sigma_t$ is fixed near the boundary. 

\begin{lma}
\label{lma:sympiso}
The Hamiltonian $H_t \colon X \to \R$ for which $\phi^t_{H_t}(\Sigma_0)=\Sigma_t$ can be taken to vanish along all of $\Sigma_t.$ Hence, after a deformation by a suitable family of cut-off functions, we may assume that the uniform norm of $H_t$ is arbitrarily small on all of $X.$
\end{lma}
\begin{proof}
 The Hamiltonian is constructed in the proof of \cite[Proposition 0.3]{Siebert:OnTheHolomorphicity}. 

\end{proof}

Assume that the cusp-edge of the front projection of $\Lambda_{\OP{St}}(1)$ lives above the unit circle in the $(x_1,x_2)$-plane; see Figure \ref{fig:standardsphere}.
Let $\Pi \colon \R^5 \to \R^4=\{(x_i,y_i)\}$ denote the canonical projection.

We begin the proof of Theorem \ref{thm:DisplaceLoose} in Steps 1--3 below by constructing a (non-Legendrian) smooth and arbitrarily $C^0$-small push-off of $\Lambda_{\OP{St}}(1)$ that can be displaced by a contact Hamiltonian of any fixed small oscillation. (Compare to the result in \cite{HamiltonianDisjunction} for non-Lagrangian submanifolds.) In Step 4, we show how with a small oscillation, the stabilized sphere in the Theorem's statement is Legendrian isotopic to a Legendrian $C^0$-close to the initial push-off. We then apply the ambient isotopy of the non-Legendrian's displacement from Steps 1--3 to complete the Legendrian isotopy and the proof.

\emph{Step 1:} Arguing similarly as in the proof of Lemma \ref{lma:sharp} in Section \ref{pf:sharp}, but while taking some additional care, one can readily find a Legendrian isotopy $\Lambda_t$ from $\Lambda_0 = \Lambda_{\OP{St}}(1)$ to $\Lambda_1=\Lambda'$ for which the isotopy $\Pi(\Lambda_t)$ has support in the interior of $U \coloneqq \{(x_1,x_2) \in D^2_{1-2\epsilon} \setminus D^2_{2\epsilon}\} \subset \R^4,$ and such that $\Lambda'$ is displaceable by the lift of a Hamiltonian on $\R^4$ of very small oscillation. This is because $\Pi(\Lambda')$ may be assumed to live in a small neighborhood of the Lagrangian disc $\{(x_1,x_2,y_1,y_2) \in D^2_1\times\{0\} \} \subset \R^4.$

To live in such a small neighborhood, the Reeb chord of $\Lambda_t$ must certainly shrink; however, the Lagrangian projection $\Pi(\Lambda_t)$ still can be assumed to be fixed near the corresponding double-point. We may assume that all $\Lambda_t$ have a unique Reeb chord and that their projections to the $(x_1,x_2)$-plane are submersions inside $\Pi^{-1}(U) \subset \R^5.$

\emph{Step 2:} We now construct a $C^0$-small \emph{non-Legendrian} push-off $\widetilde{\Lambda}_t$ of $\Lambda_t$ such that
\begin{itemize}
\item the image $\Pi(\widetilde{\Lambda}_t)$ is \emph{symplectic} in the subset $U = \{(x_1,x_2) \in D^2_{1-2\epsilon} \setminus D^2_{2\epsilon}\} \subset \R^4,$
\item $\widetilde{\Lambda}_t=\Lambda_t$ in the complement of the subset $\Pi^{-1}(V)$ for $V \coloneqq \{(x_1,x_2) \in D^2_{1-\epsilon} \setminus D^2_\epsilon\} \subset \R^4,$ and
\item $\widetilde{\Lambda}_t$ does not depend on $t$ inside $\Pi^{-1}(V \setminus U).$
\end{itemize}
The deformation can be performed by, for example, considering a Lagrangian standard neighborhood of $\Pi(\Lambda_t) \subset \R^4$ and pushing it off as a section consisting of a suitable family of one-forms whose exterior derivative is a symplectic form on $\Lambda_t$ inside $U.$ Note that this deformation $\widetilde{\Lambda}_t$ necessarily must be anti-symplectic somewhere inside $V \setminus U$ by Stokes' theorem: a closed chain inside $\R^4$ cannot have nonzero symplectic area. However, the family of one-forms can still be taken to be fixed inside $V \setminus U,$ which ensures the third bullet point above. 

\emph{Step 3:} We can now apply Lemma \ref{lma:sympiso} inside $U \subset \R^4$ to $\Pi(\widetilde{\Lambda}_t)$ in order to deduce the existence of a Hamiltonian on $\R^4$ of arbitrarily small oscillation that generates the isotopy $\Pi(\widetilde{\Lambda}_t) \subset \R^4.$ Since $\widetilde{\Lambda}_1$ is assumed to be arbitrarily $C^0$-close to $\Lambda',$ and $\Lambda'$ is displaceable by the lift of a Hamiltonian on $\R^4$ having small oscillation, this finishes the construction of the push-off with a displacement of small oscillatory norm.

\emph{Step 4:}
Consider a standard contact neighborhood $(T^*_{\le \epsilon}S^2 \times [-\epsilon,\epsilon],dz-\mathbf{p}d\mathbf{q})$ of $\Lambda_{\OP{St}}(1)$ in which $\Lambda_{\OP{St}}(1)$ corresponds to the zero section, and which contains the non-Legendrian sphere $\widetilde{\Lambda}_0$ which can be displaced with small oscillation. Here we may assume that $\epsilon>0$ is arbitrarily small.

We add a stabilization to $\Lambda_{\OP{St}}(1)$ inside the above standard neighborhood to create a loose Legendrian $\Lambda_{\OP{loose}}$ in the same formal isotopy class as $\Lambda_{\OP{St}}(1),$ which is the unique 2-sphere formal isotopy class \cite{Murphy:Loose}. By Murphy's $h$-principle \cite{Murphy:Loose} we can find a Legendrian isotopy confined to the above standard neighborhood that takes $\Lambda_{\OP{loose}}$ to a Legendrian that approximates $\widetilde{\Lambda}_0$ arbitrarily well in the $C^0$-norm; here we need that $\widetilde{\Lambda}_0$ admits a $C^0$ approximation by a Legendrian sphere \cite{Dan} and the aforementioned uniqueness of the formal isotopy class. What remains is to argue that the oscillation of this Legendrian isotopy can be assumed to be of order $\epsilon.$

This can be achieved by applying a fiber-wise rescaling by a small positive number to the whole isotopy, thereby making it confined to an arbitrarily small neighborhood of the zero-section. We then just need to estimate how much oscillation is needed to do the initial shrinking of $\Lambda_{\OP{loose}},$ together with the expansion back to the approximation of $\widetilde{\Lambda}_0.$ The crucial estimates of the oscillation of the fiber-wise rescaling, i.e.~the contact isotopy
$$(\mathbf{q},\mathbf{p},z) \mapsto (\mathbf{q},(1-t)\mathbf{p},(1-t)z),\:\:t \in [0,1),$$
were considered in the proof of Lemma \ref{lma:sharp} in Section \ref{pf:sharp}. Its generating contact Hamiltonian with respect to the standard tautological contact form $dz-\mathbf{p}d\mathbf{q}$ is given by $H_t=-\frac{z}{1-t}.$ The $\epsilon$-neighborhood of the zero-section can thus be shrunk to a $\lambda \cdot \epsilon$-neighborhood, where $0<\lambda<1,$ with a contact Hamiltonian of oscillation $2(1-\lambda)\epsilon.$ 

\qed

\bibliographystyle{plain}
\bibliography{references}
\end{document}